\documentclass[smallcondensed]{svjour3}

\usepackage{amsfonts,latexsym}
\usepackage{epsfig}
\usepackage{amsmath}
\usepackage{amssymb}
\usepackage{graphicx}
\usepackage{caption}
\usepackage{subfig}
\usepackage{sidecap}
\usepackage{caption}
\usepackage{color}

\usepackage[latin1]{inputenc}
\usepackage[english]{babel}
\usepackage{wasysym}
\usepackage{bm}
\usepackage{appendix}
\usepackage[latin1]{inputenc}
\usepackage[english]{babel}

\begin{document}
\title{Space-Time  Decomposition of Kalman Filter 
}

\author{Luisa D'Amore, Rosalba Cacciapuoti}
\institute{Luisa D'Amore  \and Rosalba Cacciapuoti
\at
            University of Naples Federico II, Naples, Italy\\
            \email{(luisa.damore,rosalba.cacciapuoti)@unina.it} ORCID: https://orcid.org/0000-0002-3379-0569}  
\maketitle

\renewcommand{\thefootnote}{\fnsymbol{footnote}}

\begin{abstract}
  We present an innovative interpretation of Kalman Filter (KF, for short) combining  the ideas of Schwarz Domain Decomposition (DD) and Parallel in Time (PinT)  approaches. Thereafter  we call it  DD-KF. In contrast to  standard DD approaches  which are already incorporated in KF and other state estimation models, implementing a straightforward data parallelism inside the loop over time, DD-KF  \emph{ab-initio} partitions  the whole model,  including  filter equations and dynamic  model along both space and time directions/steps. As a consequence,  we  get  local KFs  reproducing the original filter at smaller dimensions on local domains. Also,  sub problems could be solved in parallel.  In order to enforce the matching of local solutions
on  overlapping regions, and then  to  achieve the same global solution of KF,   local KFs are slightly modified by adding a
correction term keeping track of contributions of adjacent subdomains to  overlapping regions. Such a correction term balances localization errors along overlapping regions, acting as a regularization constraint on local solutions. Furthermore,  such a localization excludes remote observations from each analyzed location improving the conditioning of the error covariance matrices. 
As dynamic model we consider Shallow Water equations which can be regarded a consistent tool to get a proof of concept of the reliability assessment  of DD-KF  in monitoring and forecasting of weather systems and ocean currents. \end{abstract}

AMS:35Q93,49M27, 49M41, 65K15, 65M32

\section{Introduction}
\label{sec:1}

\noindent In the present paper we focus on Kalman Filter (KF) which is one of the most important and common  state estimation algorithms solving Data Assimilation (DA) problems \cite{Kalman,Kalnay}. Besides its employment in  validation of  mathematical models used in meteorology, climatology, geophysics, geology and hydrology \cite{Wu}, it  has become a main component in visual tracking of moving object in image 
 processing or   autonomous vehicles tracking with GPS in satellite navigation systems, or even in applications of non physical systems such as  financial markets \cite{Karavasilis,Benhamou,Heidari,Wu}.    Its main strength is the simple derivation of the algorithm using the filter recursive property:  measurements/observations are processed step by step as they arrive allowing to correct the given data. A recent review of KF and its applications is given in \cite{Kim}.\\
\noindent KF  time complexity
requires very large computational burden in concrete scenarios  \cite{Lyster,Tang,Tossavainen}.  Hence, several variants/approximation  have been proposed to reduce the computational complexity and to allow KF deployment in real time. These are  designed on the basis
of a reduction in the order of the  model  \cite{Hahnel,Hannachi,Rozier,Wikle},  or they are based on { Ensemble Kalman Filter (EnKF)}  which approximates KF by representing the distribution of the state with an ensemble of draws from that distribution. A prediction of the error at a future time is computed by integrating each ensemble state independently by the model \cite{Evensen}. EnKF can be used for nonlinear dynamics settings and they are  more suited to large scale real world problems than the KF, so it is more viable for use within operational DA \cite{Sandu}. Nevertheless, EnKF methods are hindered by a reduction in the size of the ensemble due to
the computational requirements of maintaining a large ensemble \cite{Anderson,Hamill}.  In applications, most efforts to deal with the related problems of computational effort and sampling error in ensemble estimation have focused on  using variations on the concept of localization \cite{Zhou}.\\
\noindent A typical approach for solving computationally intensive problems - which is  oriented to exploit parallel computing - is based on Schwarz Domain Decomposition (DD) techniques. DD methods are  well-established strategies that has been used with great success for a wide variety of problems. Review papers, both for mathematical and computational analysis are for instance  \cite{Meurant,Quarteroni,Saad,Schwarz}. Concerning DA problems, DD is usually performed by discretizing the objective function and the model to build a discrete Lagrangian \cite{DD-DA}. At the heart of these so called all-at-once approaches,  lies  solution of a very large linear system (the Karush-Kuhn-Tucker (KKT) system) which is  decomposed according to the discretization grid. 
A common drawback of such parallel algorithms is their limited scalability, due to the fact that parallelism is achieved adapting  the most computationally demanding tasks  for parallel execution. Concerning KF  this approach leads to the straightforward data parallelism inside the loop over time-steps.  
Amdhal's Law clearly applied in these situations because the computational cost of the components that are not parallelized - or in case of KF a data synchronization at each step -  provides a lower bound on the execution time of the parallel algorithm.  More generally if a fraction $r$  of our serial program remains unparallelized, then
Amdahl's Law says we can't get a speedup better than  $1/r$. Thus even if r is quite small, say 1/100, and we have a system with thousands of cores, we can't get a speedup better than 100.

\noindent Here we present the mathematical framework of an innovative approach which turns to be appropriate for using state estimation problems described by KF and governed by a dynamic model such as Partial Differential Equations (PDEs).  This approach  relies on the ideas of  Schwarz DD  \cite{Schwarz} and Parallel in Time (PinT) approaches \cite{Lions} fully revised in a linear algebra setting.  Schwarz methods  use as boundary conditions of the local PDE-model the approximation of the numerical solution computed on the interfaces between adjacent spatial subdomains. In order to introduce a consistent DD along the time direction,  PinT methods use a coarse/global/predictor propagator to obtain approximate initial values of local models on the coarse time-grid;  a fine/local/corrector solver to obtain the solution of local models starting from the approximate initial values; 
 an iterative procedure to smooth out the discontinuities of the global model on the overlapping domains. 
Nevertheless, one of the key limitation  of any PinT-based methods is data dependencies of the local solvers from the coarse solver: the coarse solver must always be executed serially for the full duration of the simulation and local solver have to wait for the approximate initial values provided by the coarse solver. We  use the KF solution  as coarse predictor for the local PDE-based  model, providing the approximations needed  for locally solving  the initial value problems on each time subinterval,  whereas the PDE model serves as a fine description which locally improves the coarse one, on each time interval in order to iteratively improve the first prediction.  \\
 We underline that the idea of partitioning the spacial domain into sub-domains and use KF on these subdomains for a fixed time is quite old, as it goes back to Lorenc in 1981 \cite{Lorenc1981}. This is essentially the idea underlying domain localization which has been   deeply investigated in applications (see for instance \cite{Brankart,Brusdal,Haugen,Nerger2006,Ott}). Because the assimilations are performed independently in each local region, two neighboring subdomains might produce strongly different analysis estimates when the assimilated observations have
gaps. Our approach instead  resolves both the inherent bottleneck of time--marching solvers and the smoothness problem of the analysis fields. 

\noindent We emphasize that there is a quite different rationale behind such  DD framework and the so called Model Order Reduction  (MOR) methods, even though they are closely related each other.  
The primary motivation of DD methods based on  Additive Schwarz Method (ASM)  was the inherent parallelism arising from a flexible, adaptive  and independent decomposition of the given problem into several sub problems, though they can also reduce the complexity of sequential solvers. ASM  algorithms and theoretical frameworks  are, to date, the most mature for this class of problems. 
MOR techniques are based on projection of the full order model onto a lower dimensional space spanned by a reduced order basis. These methods has been used extensively in a variety of fields for efficient simulations of highly intensive computational problems. But all numerical issues  concerning the quality of approximation still are of paramount importance \cite{Petzold}.
As  previously mentioned, DD-KF  framework makes it natural to switch from a full scale solver to a model order reduction solver for solution of subproblems for which no relevant low-dimensional reduced space should be constructed. In the same way, DD-KF framework leads to the model decomposition in space and time which is coherent with the filter localization. In conclusion,  main advantage of the DD-KF  is to combine in one theoretical framework,  model decomposition, along the space dimension and time steps, and filter itself, while  providing a flexible,  adaptive, reliable and robust decomposition.  

\noindent Summarizing,  we partition initial domain  along space dimension and time steps,  then we extend each subdomain to overlap its neighbors by an amount; partitioning can be adapted according to  the availability of measurements and data \cite{DyDD}. Accordingly, we decompose forecasting model both in space and time. In particular, it means that as initial and boundary values of local forecasting  models, according to Parallel in Time and Schwarz domain decomposition idea, we  use  estimates provided by  KF  itself at previous step in the adjacent space-time interval, as soon as these are computed.  On each subdomain we formulate a local KF problem analogous to the original one, defined on local  models. In order to enforce the matching of local solutions on  overlapping regions, local KF problems are slightly modified by adding a correction term, acting as a smoothness-regularization constraint on local solutions, which keeps track of contributions of adjacent domains to  overlapping regions; the same correction is applied to covariance matrices,  thereby improving the conditioning of the error covariance matrices.  \\
    To the best of our knowledge, such  an ab-initio decomposition of KF along space and time steps has never been investigated before. A spatially distributed KF into sensor based reduced-order models, implemented on a sensor networks where multiple sensors are mounted on a robot platform for target tracking, is presented in \cite{Battistelli,Khan}. A new algorithm only using domain localization in Extended Kalman Filter is proposed in \cite{Janic}.   

\noindent We would like to mention that in the present work we derive and discuss main features of DD-KF framework using the Shallow Water Equations (SWEs) which are commonly used for monitoring and forecasting the water flow
in rivers and open channels (see for instance \cite{Rafiee,Tirupachi}).  Nevertheless, as we follow the first discretize then optimize approach, Schwarz and PinT methods are completely revisited in the linear algebra setting of this framework, allowing its application in the plentiful  literature of  state estimation real-world problems. We observe that the same authors have presented a simplified version of DD-KF in \cite{PPAM2019}, where KF was used for solving  CLS (Constrained Least Square) models,
seen as  prototype model of Data Assimilation problems. CLS is obtained combining two overdetermined linear systems, representing the state and the
observation mapping. \\
\noindent The rest of the article is organized as follows. In section \S 2, we review preliminaries on KF then in section \S 3 we describe   DD-KF framework and its related algorithm. In section 4, we define the mathematical framework underlying DD-KF. In section \S 5, we theoretically prove its reliability  while  validation on the Shallow Water Equations are given in section \S 6.  Conclusions and possible extensions are provided in section \S 7. Finally, with the aim of allowing replicability of  experimental results section \S 8 contains  the whole detailed derivation of DD-KF method on  SWEs and performance analysis of DD-KF w.r.t. KF.

\normalsize


\section{Preliminaries on Kalman Filter}
Given $x_{0}\in \mathbb{R}^{N}$, let $x(t)\in \Omega \subset \mathbb{R}^{N}$, where $\forall t \in \Delta$, denote the state of a dynamic system governed by the mathematical model $\mathcal{M}_{t,t+\Delta t}( x(t))$, such that:
\begin{equation}\label{modello}
\left\{\begin{array}{ll}
x(t+ \Delta t)&=\mathcal{M}_{t, t+ \Delta t}(x(t)), \\
x(0)&=x_{0}
\end{array}, \right.
\end{equation}  
and let:
\begin{equation}\label{osservazioni}
y(t+\Delta t)=\mathcal{H}_{t+\Delta t}[x(t+\Delta t)],
\end{equation}
denote the so called observations, where 
\begin{equation}
\mathcal{H}_{t+\Delta t}: \ x(t+\Delta t) \in \mathbb{R}^{n}\mapsto y(t+\Delta t) \in \mathbb{R}^{d}\,, \quad d \in \mathbb{N} \quad d <<N \, ,
\end{equation}
denotes the observations mapping including transformations and grid interpolations.  Here we may assume that both $\mathcal{M}$  and $\mathcal{H}$ are non linear. \\
In the context of DA methods, KF aims to bring the state $H(x(t))$ as close as possible to the measurements/observations $y(t)$. One can do this by discretizing  then optimize or first optimize and then discretize. Here, with the aim of making the best use of Schwarz and PinT methods in a linear algebra settings,  we  first discretize then optimize. \\ \noindent Let $r\in \mathbb{N}$ be given. We consider $r+2$ points in $\Delta$, let us say $\{t_{k}\}_{k=0,1,\ldots,r+1}$,  where $t_{k}=k\Delta t$ and $\Delta t= \frac{T}{r+1}$ (more in general, the mesh $t_k$ need not be uniform). \\ \noindent  We use the following setup of  KF: 
\noindent   $x_{k}\equiv x(t_{k})$,  state  at $t_{k}$; 
$ x_{0}$:  state at  $t_{0}\equiv 0$ (or the vector of initial conditions);
$M_{k,k+1}\in \mathbb{R}^{N\times N}$:  discretization of the tangent linear operator (Jacobian) of $\mathcal{M}_{t,t+\Delta t}( x(t))$;
$m\in \mathbb{N}$:  number of observations;
$y_{k}\equiv y(t_{k})\in \mathbb{R}^{m\cdot d}$:  observations vector;
$b_{k}\in \mathbb{R}^{ N}$: input control vector;
$H_{k+1}\in \mathbb{R}^{m\cdot d \times N}$:   discretization of tangent linear approximation of $\mathcal{H}_{t_{k+1}}$ (also called observations operator)  with $N>m$, for $k=0,1,\ldots,r$;
$w_{k}\in \mathbb{R}^{n}$ and $v_{k}\in \mathbb{R}^{m\cdot d}$: model and observation additive errors with normal distribution and zero mean;
$Q_{k}\in \mathbb{R}^{N\times N}$ and $R_{k}\in \mathbb{R}^{m\cdot d \times m\cdot d}$: covariance matrices of the errors on the model and on the observations, such that 
\begin{equation}
Q_{k}:=E[w_{k}w_{k}^{T}] \quad R_{k}:=E[v_{k}v_{k}^{T}] \quad \textit{$\forall$ $k=0,1,\ldots,r+1$}.
\end{equation}
\noindent where $E[\cdot]$ denotes the expected value, $\forall k=0,1,\ldots,r$.  These matrices are symmetric and positive definite.

\noindent KF consists in the iterative  estimate  of  $x_{k+1}\in \mathbb{R}^{n}$:
\begin{equation}\label{sistema_kalmandiscreto}
x_{k+1}=M_{k,k+1}x_{k}+b_{k}+w_{k}, \quad 
\end{equation}
such that 
\begin{equation}\label{problema1}
y_{k+1}=H_{k+1}{x}_{k+1}+v_{k+1}\quad .
\end{equation}
An example of the vector $b_k$ is given for the Shallow Water Equations described in the Appendix (see (\ref{b[1]})).\\

\noindent  
If $P_{0}\in \mathbb{R}^{n\times n}$ and  $\widehat{x}_{0}:= x_{0}$, each step of  KF  is composed by two main operations \cite{sorenson}.
\begin{itemize}
\item[(i)] Predictor phase, involving computation of  state estimate:
\begin{equation}\label{stimapredetta}
x_{k+1}=M_{k,k+1}\widehat{x}_{k}+b_{k};
 \end{equation}
computation of  covariance matrices: 
\begin{eqnarray}\label{predictedmatrix}
P_{k+1}=&M_{k,k+1}P_{k}M_{k,k+1}^{T}+Q_{k}\\ \nonumber
S_{k+1}=&H_{k+1}P_{k+1}H_{k+1}^{T}+R_{k+1} 
\end{eqnarray}
and computation of   KF gain:
\begin{equation}\label{guadagnodikalman}
K_{k+1}=P_{k+1}H_{k+1}^{T}S_{k+1}^{-1}.
\end{equation}
\item[(ii)] Corrector phase, involving 
the update of covariance matrix:
\begin{equation}\label{updateP}
P_{k+1}=(I-K_{k+1}H_{k+1})P_{k+1},
\end{equation}
and the update of  state estimate:
\begin{equation}\label{stimapredetta1}
\widehat{x}_{k+1}=x_{k+1}+K_{k+1}(y_{k+1}-H_{k+1}x_{k+1}).
\end{equation}
\end{itemize}

\section{DD-KF}
 Let  $D_{n}(\Omega)=\{{x}_{j}\}_{j=1,\ldots,n}$, be the mesh partitioning of $\Omega$ and  $b_c\in \mathbb{N}$ such that  $D_{b_c}(\partial \Omega)=\{{z}_{p}\}_{p=1,\ldots,b_c}$ be the discretization of $\partial \Omega$. If we let $x_{k+1}^{\partial \Omega}\in \mathbb{R}^{b_c}$ to be the state at  $t_{k+1}$ on $\partial \Omega$, it is
\begin{equation}
x_{k+1}^{\partial \Omega} \equiv x^{\partial \Omega}(t_{k+1})\in \mathbb{R}^{b_c}.
\end{equation} 
\noindent DD-KK  consists of  two main steps, namely domain decomposition and model decomposition. In the following, for simplicity of notation, and without loss of generality  we consider two spatial subdomains, while  the general case involving more than two subdomains will be briefly described in section 3.1.1. \\[.1cm]

\noindent \textbf{DD in space}.  Let $\Omega_{1}$, $\Omega_{2}\subset \Omega$ such that
\begin{equation}\label{subdomains_space}
\Omega= \Omega_{1}\cup \Omega_{2}
\end{equation}
with $D_{n_{1}}(\Omega_{1}),D_{n_{2}}(\Omega_{2}) \subset D_{n}(\Omega)$, where $n_{i}<n$ are such that 
\begin{equation}\label{disc_1e2}
D_{n_{1}}(\Omega_{1})=\{x_{i}\}_{i\in I_{1}} \quad , \quad D_{n_{2}}(\Omega_{2})=\{x_{i}\}_{i\in I_{2}}
\end{equation}
and  $n_{1}=|D_{n_{1}}(\Omega_{1})|$, $n_{2}=|D_{n_{2}}(\Omega_{2})|$; let 
\begin{equation}\label{overlap_domain}
\Omega_{1,2}:=\Omega_{1}\cap \Omega_{2}, 
\end{equation}
be the overlap region where we introduce 
\begin{equation}\label{disc12}
D_{s}(\Omega_{1,2}):=D_{n_{1}}(\Omega_{1})\cap D_{n_{2}}(\Omega_{2})=\{x_{i}\}_{i\in I_{1,2}}
\end{equation}
denoting the discretization of $\Omega_{1,2}$ and
\begin{equation}\label{set_indici}
I_{1}=\{1,\ldots,n_{1}\} ,\quad I_{2}=\{n_{1}-s+1,\ldots,n\}, \quad I_{1,2}=\{n_{1}-s+1,\ldots,n_{1}\}
\end{equation}
are the corresponding index sets, with $s\in \mathbb{N}_{0}$.\\

\noindent \textbf{DD in time}. Let  $\Delta_{j}\subset \Delta$ such that
\begin{equation}
\Delta=\cup_{j=1}^{L} \Delta_{j}\,.
\end{equation}
Let $D_{s_{j}}(\Delta_{j}) \subset D_{r+2}(\Delta)$ be  such that 
\begin{equation}\label{discreto_tempo}
D_{s_{j}}(\Delta_{j})=\{t_{k}\}_{k=\bar{s}_{j-1},\ldots,\bar{s}_{j-1}+s_{j}}
\end{equation}
denoting the discretization of $\Delta_j$, 
where in (\ref{discreto_tempo})
the quantity 
\begin{equation}\label{essesign}
    \bar{s}_{j-1}:=\sum_{l=1}^{j-1}(s_{l}-s_{l-1,l}), \quad  s_{j-1,j}\in \mathbb{N}_{0}
\end{equation} is the number of elements in common between  subsets $\Delta_{j-1}$ and $\Delta_{j}$, and it is such that $s_{0}:= 0$, $s_{0,1}:= 0$ and $\bar{s}_{L-1}+s_{L}:= r+1$,
then
\begin{equation}
\Delta_{j-1,j}:=\Delta_{j-1}\cap \Delta_{j}, 
\end{equation}
with $$D_{s_{j-1,j}}(\Delta_{j-1,j}):=D_{s_{j-1}}(\Delta_{j-1})\cap D_{s_{j}}(\Delta_{j})$$ denoting the discretization of the overlapping time interval. \\

\noindent The innovative idea is to split the outer loop of KF (let us say, the time - step loop along $k=1, \ldots, r$)  in $L$ parts each one running in $\Delta_j$ (where $j=1, \ldots,L$) as described below. \\

\noindent  For $j=1, \ldots, L$, 
  let us indicate by  $\widehat{x}_{k+1}^{\Delta_j}$, the KF solution  such that:
\begin{equation}
\widehat{x}_{k+1}^{\Delta_{j}}=M_{k,k+1}\widehat{x}_{k}^{\Delta_{j}}+b_{k}+w_{k}, \quad \forall k=\bar{s}_{j-1},\ldots,\bar{s}_{j-1}+s_{j}-1
\end{equation}
where 
\begin{equation}\label{cond_ini}
\widehat{x}_{\bar{s}_{j-1}}^{\Delta_{j}}=\widehat{x}_{\bar{s}_{j-1}}^{\Delta_{j-1}}
\end{equation} 
is the  initial value  such that (according to DA condition):
\begin{equation}
y_{k+1}=H_{k+1}\widehat{x}_{k+1}^{\Delta_{j}}+v_{k+1},
\end{equation}
in particular it is $\widehat{x}_{0}^{\Delta_{0}}:= x_{0}$. We note in (\ref{cond_ini}) that according to  PinT compatibility conditions initial value at $\Delta_j$ is  the estimate computed at previous step in $\Delta_{j-1}$. \\

\noindent We now consider the following block partitioning of   $M:= M_{k,k+1}\in \mathbb{R}^{n\times n}$ and  $H_{k+1}\in \mathbb{R}^{m\cdot d \times n}$: 
\begin{equation}\label{noverlapM}
M=\left[\begin{array}{c|c}
M_{1} & M_{1,2}\\
\hline
M_{2,1} & M_{2}
 \end{array} \right],
\end{equation}
and 
\begin{equation}\label{noverlap-H}
    H_{k+1}=\left[H_{k+1}^{1} \ H_{k+1}^{2}\right].
\end{equation}

\noindent In particular, 
\begin{itemize}
\item  if $s=0$ (no overlap), for $i,j=1,2$:
\begin{equation}\label{deco_M}
M_{i}:=M|_{I_{i}\times I_{i}}\in \mathbb{R}^{n_{i}\times n_{i}}\quad M_{i,j}:=M|_{I_{i} \times I_{j}}\in \mathbb{R}^{n_{i}\times n_{j}},
\end{equation}
    and 
    \begin{equation}
        H_{k+1}^{1}:=H_{k+1}|_{I_{1}}\in \mathbb{R}^{m\cdot d \times n_{1}}, \quad H_{k+1}^{2}:=H_{k+1}|_{I_{2}}\in \mathbb{R}^{m\cdot d \times n_{2}},
    \end{equation}

\item while if $s\neq 0$ (in presence of overlapping region), we have:
\begin{displaymath}
\begin{array}{lllllll}
\mathbb{M}_{1,1}&:=M|_{\tilde{I}_{1} \times \tilde{I}_{1}}\in \mathbb{R}^{(n_{1}-s)\times (n_{1}-s)}, \quad
\mathbb{M}_{1,2}&:=M|_{\tilde{I}_{1} \times {I}_{1,2} }\in \mathbb{R}^{(n_{1}-s)\times s},\\
\mathbb{M}_{1,3}&:=M|_{\tilde{I}_{1} \times \tilde{I}_{2} }\in \mathbb{R}^{(n_{1}-s)\times (n_{2}-s)},\quad

\mathbb{M}_{2,1}&:=M|_{{I}_{1,2} \times \tilde{I}_{1}  }\in \mathbb{R}^{s\times (n_{1}-s)},\\
\mathbb{M}_{2,2}&:=M|_{{I}_{1,2} \times {I}_{1,2} }\in \mathbb{R}^{s \times s}, \quad \ \ \ \ \ \ \ \ \ \ \
\mathbb{M}_{2,3}&:=M|_{{I}_{1,2}\times \tilde{I}_{2} }\in \mathbb{R}^{s\times (n_{2}-s)},\\

\mathbb{M}_{3,1}&:=M|_{ \tilde{I}_{2} \times \tilde{I}_{1} }\in \mathbb{R}^{(n_{2}-s)\times (n_{1}-s)}, \quad
\mathbb{M}_{3,2}&:=M|_{ \tilde{I}_{2} \times {I}_{1,2} }\in \mathbb{R}^{(n_{2}-s)\times s}, \\
\mathbb{M}_{3,3}&:=M|_{ \tilde{I}_{2} \times \tilde{I}_{2}  }\in \mathbb{R}^{(n_{2}-s)\times (n_{2}-s)}.
\end{array}
\end{displaymath}

Then  in (\ref{noverlapM}) we have:  \begin{equation}\label{deco_M2}
\begin{array}{lllllllllll}
M_{1}:=\left[\begin{array}{cc}
\mathbb{M}_{1,1} & \mathbb{M}_{1,2}\\
\mathbb{M}_{2,1}  & \mathbb{M}_{2,2} 
 \end{array} \right], \quad {M}_{1,2}:=\left[\begin{array}{cc}
0 & \mathbb{M}_{1,3}\\
0  & \mathbb{M}_{2,3} 
 \end{array} \right],\\
{M}_{2,1}:=\left[\begin{array}{cc}
\mathbb{M}_{2,1} & 0\\
\mathbb{M}_{3,1}  & 0 
 \end{array} \right], \quad M_{2}:=\left[\begin{array}{cc}
\mathbb{M}_{2,2} & \mathbb{M}_{2,3}\\
\mathbb{M}_{3,2}  & \mathbb{M}_{3,3} 
 \end{array} \right],
 \end{array}
\end{equation}
and 
\begin{equation}\label{mat_H}
    H_{k+1}^{1}:=\left[H_{k+1}^{1,1} \ \ \alpha H_{k+1}^{1,2}\right], \quad  H_{k+1}^{2}:=\left[\beta H_{k+1}^{1,2} \  \  H_{k+1}^{2,2}\right] ,
\end{equation}
 with:
\begin{equation}
    H_{k+1}^{1,1}:=H_{k+1}|_{\tilde{I}_{1}}, \quad H_{k+1}^{2,2}:=H_{k+1}|_{\tilde{I}_{2}}, \quad H_{k+1}^{1,2}:=H_{k+1}|_{{I}_{1,2}},  
\end{equation}
where 
\begin{equation}\label{set_indici_tilde}
\tilde{I}_{1}:=I_{1} \setminus  {I}_{1,2}, \quad \tilde{I}_{2}:=I_{2} \setminus  {I}_{1,2},
\end{equation}
and $n_{1}-s=|\tilde{I}_{1}|$, $n_{2}-s=|\tilde{I}_{2}|$.  We note that in (\ref{mat_H}), in presence of overlapping region, it should be $ \alpha H_{k+1}^{1,2}+ \beta H_{k+1}^{1,2}= H_{k+1}|_{{I}_{1,2}}$, so we require that  $\alpha + \beta = 1$.
\end{itemize}

\noindent We highlight that  DD-KF  employs  a block partitioning of $M$ as in the Additive  Schwarz DD, while as in the PinT approach, initial value of  restricted models are the KF estimates computed  at previous step.     In particular, at $t_{0}\equiv 0$, DD-KF in $\Delta_{1}$ employs as initial value $x_{0}\equiv \widehat{x}_{0}$ which is the initial condition of the dynamic model appropriately restricted to  $\Omega_{1}$ or $\Omega_{2}$.\\ 

\noindent In Fig. \ref{fig_DD} we report a schematic description of how  DD-KF trajectory is decomposed  in two spatial subdomains and six time subdomains. 
\begin{center}
\begin{figure}[h!]
{\includegraphics[width=1.2\textwidth]{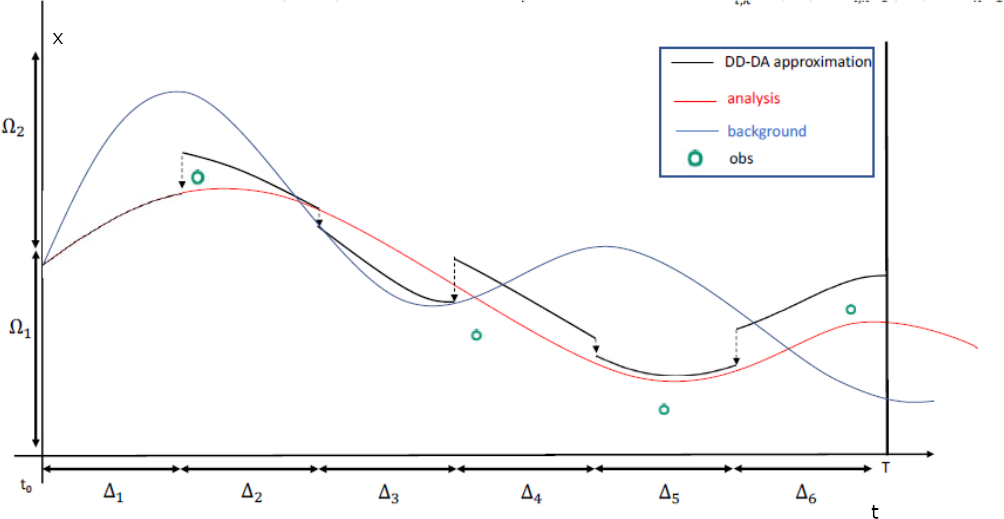}}
\caption{Schematic description of DD-KF solution trajectory for $N=1$, that is $\Omega $ is a subset of $\Re$ and is splitted in two sub domains (represented along the x-axis), while the time interval   is splitted in six sub intervals (represented along the t-axis).}
\label{fig_DD}
\end{figure}
\end{center}
\noindent That said, we now consider local KF problems in  $\Delta_{j}\subset \Delta$.  \\
\noindent \begin{definition}
(Local KF problems)
\noindent KF problems in  $\Delta_{j} \times \Omega_1$ and $\Delta_{j} \times \Omega_2$ are, $\forall \, k=\bar{s}_{j-1},\ldots,\bar{s}_{j-1}+s_{j}-1$, the following:
\begin{equation}\label{subproblemsp1}
P_{\Omega_{1}}^{\Delta_{j}}: \ \ \left\{ \begin{array}{ll}
 x_{1,k+1}^{\Delta_{j}}=M_{1}x_{1,k}^{\Delta_{j}}+b_{k}|_{I_{1}}+b_{1,k}+w_{k}|_{I_{1}}\\
y_{k+1}=H_{k+1}^{1}{x}_{1,k+1}^{\Delta_{j}}+H_{k+1}^{2}{x}_{2,k+1}^{\Delta_{j}}+v_{k+1}
\end{array}\right. 
\end{equation}
and 
\begin{equation}\label{subproblemsp2}
P_{\Omega_{2}}^{\Delta_{j}}: \ \             \left\{ \begin{array}{ll}
x_{2,k+1}^{\Delta_{j}}=M_{2,1}x_{1,k}^{\Delta_{j}}+b_{k}|_{I_{2}}+b_{2,k}+w_{k}|_{I_{2}}\\
y_{k+1}=H_{k+1}^{1}{x}_{1,k+1}^{\Delta_{j}}+H_{k+1}^{2}{x}_{2,k+1}^{\Delta_{j}}+v_{k+1}
\end{array}, \right.
\end{equation}
with initial states 
\begin{equation}\label{initial_states}
x_{1,\bar{s}_{j-1}}^{\Delta_{j}}=x_{1,\bar{s}_{j-1}}^{\Delta_{j-1}}\quad  x_{2,\bar{s}_{j-1}}^{\Delta_{j}}=x_{2,\bar{s}_{j-1}}^{\Delta_{j-1}}.
\end{equation}
In particular it is $x_{1,\bar{s}_{0}}^{\Delta_{0}}:=x_{0}|_{I_{1}}$, $x_{2,\bar{s}_{0}}^{\Delta_{0}}:=x_{0}|_{I_{2}}$ and  
\begin{equation}
\begin{array}{ll}
x_{1,k}^{\Delta_{j}}|_{\partial \Omega_{1} \setminus \Omega_{2}}={x}_{k}^{\partial \Omega_{1} \setminus \Omega_{2}}\\
x_{1,k}^{\Delta_{j}}|_{\Gamma_{1}} = x_{2,k}|_{\Gamma_{1}}
  \end{array}\quad 
  \begin{array}{ll}
x_{2,k}^{\Delta_{j}}|_{\partial \Omega_{2} \setminus  \Omega_{1}}={x}_{k}^{\partial \Omega_{2} \setminus \Omega_{1}}\\
x_{2,k}^{\Delta_{j}}|_{\Gamma_{2}}  = x_{1,k}|_{\Gamma_{2}} 
  \end{array}
\end{equation}
where
\begin{equation}
    \Gamma_{1}:=\partial \Omega_{1}\cap \Omega_{2} \quad, \quad \Gamma_{2}:=\partial \Omega_{2}\cap \Omega_{1}\\
\end{equation}
and 
\begin{equation}
    b_{1,k}:=\left[\begin{array}{ll} \mathbb{M}_{1,3}\\ \mathbb{M}_{2,3}\end{array}\right]x_{2,k}^{\Delta_{j}}|_{\Gamma_{1}} \quad b_{2,k}:=\left[\begin{array}{ll} \mathbb{M}_{2,1}\\ \mathbb{M}_{3,1}\end{array}\right]x_{1,k}^{\Delta_{j}}|_{\Gamma_{2}}.
\end{equation}
$H_{k}^{1}\in \mathbb{R}^{m\cdot d \times n_{1}}$ and $H_{k}^{2}\in \mathbb{R}^{m \cdot d \times n_{2}}$ are defined in (\ref{mat_H}) and,  for $i=1,2$, $x_{\bar{s}_{j-1}}^{\Delta_{j-1}}|_{I_{i}}$ and $w|_{I_{i}}$ are  reductions to  $I_{i}$ of  $\widehat{x}_{0}\in \mathbb{R}^{n}$ and $w\in \mathbb{R}^{n}$.
\noindent As a consequence,  the errors $$e_{1,k+1}:=(x_{1,k+1}-x_{1,k})\in\mathbb{R}^{n_{1}}, \quad  e_{2,k+1}:=(x_{2,k+1}-x_{1,k})\in\mathbb{R}^{n_{2}}$$ defined  in $\Omega_{1}$ and $\Omega_{2}$,  give the local estimate of  $$P_{i,j}:=Cov(e_{i,k+1},e_{j,k+1})$$ with  $i\neq j$.
Similarly, 
\begin{equation}
\begin{array}{ll}
Q_{k}&=diag(Q_{1,k},Q_{2,k})\in \mathbb{R}^{n\times n}\\ 
\end{array}
\end{equation}
and $D_{k}\in \mathbb{R}^{m\cdot d\times m\cdot d}$ are local covariance  matrices of  errors on model and on  observations where,  $Q_{1,k}\in \mathbb{R}^{n_{1}\times n_{1}}$ and $Q_{2,k}\in \mathbb{R}^{n_{2}\times n_{2}}$, while $D_{k}$ is a diagonal matrix.

\end{definition}
\noindent We note that the local aspect of the KF problem given in Definition 2 comes out from both the local contribution of the model and from the filter localization.

\subsection{Local problems}
In the following we let  $\Delta_{j}$ be fixed, then for simplicity of notation, we refer to $\widehat{x}_{i,k}^{\Delta_{j}}$ omitting superscript $\Delta_{j}$.
  $\forall k=\bar{s}_{j-1},\ldots,\bar{s}_{j-1}+s_{j}-1$ local problems are made of:
\begin{itemize}
\item {\bf Predictor phase}. Computation of  state estimates:
\begin{equation}\label{stimapredetta_DD}
\begin{array}{ll}
x_{1,k+1}=M_{1}\widehat{x}_{1,k}+b_{k}|_{I_{1}}+b_{1,k}+w_k|_{I_{1}}\\
x_{2,k+1}=M_{2}\widehat{x}_{2,k}+b_{k}|_{I_{2}}+b_{2,k}+w_k|_{I_{2}}
\end{array};
 \end{equation}
 where 
\begin{equation}\label{b_k}
    b_{1,k}:=\left[\begin{array}{ll} \mathbb{M}_{1,3}\\ \mathbb{M}_{2,3}\end{array}\right]\widehat{x}_{2,k}|_{\Gamma_{1}} \quad b_{2,k}:=\left[\begin{array}{ll} \mathbb{M}_{2,1}\\ \mathbb{M}_{3,1}\end{array}\right]\widehat{x}_{1,k}|_{\Gamma_{2}};
\end{equation}
computation of  the local estimate of 
\begin{equation}\label{predictedmatrix_DD}
\begin{array}{ll}
P_{1}=M_{1}P_{1}M_{1}^{T}+P_{\Omega_{1}\leftrightarrow \Omega_{2}}+Q_{1,k}\\
P_{2}=M_{2}P_{2}M_{2}^{T}+P_{\Omega_{2}\leftrightarrow \Omega_{1}} +Q_{2,k}
\end{array},
\end{equation}
where
\begin{equation}\label{matP_in}
\begin{array}{ll}
P_{\Omega_{1}\leftrightarrow \Omega_{2}} =M_{1,2}P_{2,1}M_{1}^{T}+M_{1}P_{1,2}M_{1,2}^{T}+M_{1,2}P_{2}M_{1,2}^{T}\\
P_{\Omega_{2}\leftrightarrow \Omega_{1}} =M_{2,1}P_{1,2}M_{2}^{T}+M_{2}P_{2,1}M_{2,1}^{T}+M_{2,1}P_{1}M_{2,1}^{T}
\end{array},
\end{equation}
with
\begin{equation}\label{matP}
\begin{array}{ll}
P_{1,2}=M_{1}P_{1,2}M_{2}^{T}+C_{\Omega_{1}\leftrightarrow \Omega_{2}}\\
P_{2,1}=M_{2}P_{2,1}M_{1}^{T}+C_{\Omega_{1}\leftrightarrow \Omega_{2}}^{T}
\end{array},
\end{equation}
and

\begin{equation}\label{mat_C}
\begin{array}{ll}
C_{\Omega_{1}\leftrightarrow \Omega_{2}} =M_{1}P_{1}M_{2,1}^{T}+M_{1,2}P_{2,1}M_{2,1}^{T}+M_{1,2}P_{2}M_{2}^{T}
\end{array}
\end{equation}
keeping track of contribution of  $\Omega_{1}$ and $\Omega_{2}$ to  the overlapping region.\\

\item {\bf Corrector phase}. Update of   DD-KF gains:
\begin{equation}\label{guadagnodikalman_DD}
\begin{array}{ll}
K_{1}=(P_{1}H_{k+1}|_{I_{1}}^{T}+P_{1,2}H_{k+1}|_{I_{2}}^{T})\cdot F \\
K_{2}=(P_{2}H_{k+1}|_{I_{2}}^{T}+P_{2,1}H_{k+1}|_{I_{1}}^{T})\cdot F
\end{array};
\end{equation}
where
\begin{equation}\label{matF}
    F=(H_{k+1}|_{I_{1}}P_{1}H_{k+1}|_{I_{1}}^{T}+H_{k+1}|_{I_{2}}P_{2}H_{k+1}|_{I_{2}}^{T}+R_{1,2}+R_{k+1})^{-1}
\end{equation}
and 

\begin{equation}\label{R12}
R_{1,2}=(H_{k+1}|_{I_{2}}P_{2,1}H_{k+1}|_{I_{1}}^{T}+H_{k+1}|_{I_{1}}P_{1,2}H_{k+1}|_{I_{2}}^{T})
\end{equation}
  keep track of  contribution of $\Omega_{1}$ and $\Omega_{2}$ to  overlapping region;

 update of:
\begin{equation}\label{P1-P2}
\begin{array}{ll}
P_{1}=(I-K_{1}H_{k+1}|_{I_{1}})P_{1}-K_{1}H_{k+1}|_{I_{2}}P_{2,1}\\
P_{2}=(I-K_{2}H_{k+1}|_{I_{2}})P_{2}-K_{2}H_{k+1}|_{I_{1}}P_{1,2}
\end{array};
\end{equation}
update of  local estimate of  covariance matrices between $e_{1}\in \mathbb{R}^{n_{1}}$ and $e_{2}\in \mathbb{R}^{n_{2}}$
\begin{equation}\label{P12-P21}
\begin{array}{ll}
P_{1,2}=(I-K_{1}H_{k+1}|_{I_{1}})P_{1,2}-K_{1}H_{k+1}|_{I_{2}}P_{2}\\
P_{2,1}=(I-K_{2}H_{k+1}|_{I_{2}})P_{2,1}-K_{2}H_{k+1}|_{I_{1}}P_{1}
\end{array}.
\end{equation}
Finally, we get to the update of  DD-KF estimates:
\begin{equation}\label{stimastato_DD}
\begin{array}{ll}
\widehat{x}_{1,k+1}=x_{1,k+1}+K_{1}\left[y_{k+1}-(H_{k+1}|_{I_{1}}x_{1,k}+H_{k+1}|_{I_{2}}x_{2,k})\right] \\
\widehat{x}_{2,k+1}=x_{2,k+1}+K_{2}\left[y_{k+1}-(H_{k+1}|_{I_{1}}x_{1,k}+H_{k+1}|_{I_{2}}x_{2,k})\right]
\end{array}.
\end{equation}
\end{itemize}

\noindent We underline that if  observations are concentrated in spatial subdomains such that the first $m_{1}$ and the last $m_{2}$ are in  $\Omega_{1}$ and $\Omega_{2}$, respectively,  then at  step $k+1$  these observations influence the first $n_{1}$ and the last $n_{2}$ components of DD-KF estimates $\widehat{x}_{1,k+1}$ and $\widehat{x}_{2,k+1}$, respectively. Consequently,  matrices $H_{k+1}|_{I_{i}}$  have $m-m_{i}$ null rows and then they can be reduced to $m_{i}\times n_{i}$ matrices, with $i=1,2$. \\

\noindent \textbf{Remark} (Filter Localization).  We note that the covariance localization performed by  DD-KF can be obtained using the  P. L. Houtekamer and H. L. Mitchell approach \cite{H_M}. Indeed, P. L. Houtekamer and H. L. Mitchell  made a modification to  EnKF to localize covariances matrices  and consequently the Kalman gain using the Schur product $K_{1}^{l}=[(\rho \circ (P_{k+1}H_{k+1}^{T})][ \rho \circ(H_{k+1}P_{k+1}H_{k+1}^{T})+R_{k+1}]$, where $\rho$ is 
a suitably correlation function.  If we choose   $\rho_{1}:=\mathbb{1}_{\Omega_{1}}$ and $\rho_{2}:=\mathbb{1}_{\Omega_{2}}$, i.e. the indicator functions of  $\Omega_{1}$ and $\Omega_{2}$ respectively, then we get the DD-KF gains $K_{1}$ and $K_{2}$ in (\ref{guadagnodikalman_DD}).
 
\noindent
\subsubsection{DD-KF for more than two subdomains}
 DD-KF  can be generalised to $n_{sub}>2$ spatial subdomains. Starting from the decomposition step in space and time given in section 3,  with $n_{sub}>2$ spatial subdomains consecutively arranged along a one dimensional direction. Such assumption allows to simplify the  management of the contributions of adjacent domains to  overlapping regions. In this case each subdomain overlaps with two domains and two compatibility conditions needs to be satisfied. As more than two overlapping regions occurs  as more compatibility conditions should  be satisfied. Just for simplification of the presentation we assume that   $M_{k,k+1}$ is block tridiagonal (such the matrix in (\ref{section7M}) arising from discretization of  SWEs  described in section 7) then  the decomposition step,   still holds  where instead of (\ref{noverlapM}) and (\ref{noverlap-H}) we will have respectively:
\begin{equation}\label{M_tridiag}
    M_{k,k+1}= \left[\begin{array}{ccc} \begin{array}{ccc} M_1 & M_{1,2} &  \\ M_{2,1}& M_2 &M_{2,3} \end{array} & 0 \\  0 &  \begin{array}{cc} \ddots &  \\ M_{n_{sub},n_{sub}-1}& M_{n_{sub}} \end{array} \end{array}\right]
\end{equation}
and 
\begin{equation}
    H_{k+1}=\left[H_{k+1}^{1} \ H_{k+1}^{2} \ \cdots \ H_{k+1}^{n_{sub}} \right].
\end{equation}


\noindent In the same way of (\ref{subproblemsp1}) and (\ref{subproblemsp2}), we get $n_{sub}$ local problems $P^{\Delta_j}_{\Omega_i}$, in the $i$-th subdomain, for $i=1,\ldots, n_{sub}$.  DD-KF step  at $t_{k+1}\in \Delta_{j}$, $\forall k=\bar{s}_{j-1},\ldots,\bar{s}_{j-1}+s_{j}-1$ provides  solutions of local problems   in the following steps.\\ 
\noindent Computation of state estimates:
\begin{equation*}
    x_{i,k+1}=M_h\widehat{x}_{i,k}+b_k|_{I_i}+b_{i,k}+w_k|_{I_{i}}
\end{equation*}
where 
\begin{equation*}
    b_{i,k}=\left\{\begin{array}{ll}
   M_{i,i+1}\widehat{x}_{i+1,k}|_{\Gamma_i} \quad &\textrm{if $i=1$}\\
   M_{i,i-1}\widehat{x}_{i-1,k}|_{\Gamma_i} \quad &\textrm{if $i=n_{sub}$}\\
 M_{i,i+1}\widehat{x}_{i+1,k}|_{\Gamma_i}+M_{i,i-1}\widehat{x}_{i-1,k}|_{\Gamma_i} \quad &\textrm{otherwise}\\
    \end{array}\right.;
\end{equation*}
for $h=1,\ldots,n_{sub}$, $h\neq i$ and \begin{equation}
\begin{array}{lll}
j=i+1 &if & i=1\\
j=i-1&  if& i=n_{sub}\\
j=\{i+1,i-1\} & if& 1<i<n_{sub}\\
\end{array};
\end{equation}
 computation of 
\begin{equation*}
    P_{i,h}=M_{i}P_{i,h}M_{h}^{T}+C_{\Omega_i \leftrightarrow \Omega_h} ,
\end{equation*}
and
\begin{equation*}
P_i=M_i P_i M_i^T+P_{\Omega_i \leftrightarrow \Omega_j}+Q_{i,k};
\end{equation*}
where
\begin{equation*}
C_{\Omega_i \leftrightarrow \Omega_h} =\left\{\begin{array}{ll}
\begin{array}{ll}(M_{i}P_{i,h-1}+M_{i,i+1}P_{i+1,h-1})M_{h,h-1}^{T} + M_{i,i+1}P_{i+1,h}M_{h}^{T}\end{array}  & \ \textrm{if $i=1$; $h=n_{sub}$}\\
\begin{array}{ll}(M_{i}P_{i,h-1}+M_{i,i+1}P_{i+1,h-1})M_{h,h-1}^{T} + M_{i,i+1}P_{i+1,h}M_{h}^{T}\\+(M_{i}P_{i,h+1}+M_{i,i+1}P_{i+1,h+1})M_{h,h+1}^{T}\end{array}  & \ \textrm{if $i=1$; $h\neq n_{sub}$}\\
\\
\begin{array}{ll}(M_{i}P_{i,h+1}+M_{i,i-1}P_{i-1,h+1})M_{h,h+1}^{T} + M_{i,i-1}P_{i-1,h}M_{h}^{T}\end{array} & \ \textrm{if $i=n_{sub}$; $h=1$}\\
\begin{array}{ll}(M_{i}P_{i,h+1}+M_{i,i-1}P_{i-1,h+1})M_{h,h+1}^{T} + M_{i,i-1}P_{i-1,h}M_{h}^{T}\\+(M_{i}P_{i,h-1}+M_{i,i+1}P_{i-1,h-1})M_{h,h-1}^{T} \end{array}  & \ \textrm{if $i=n_{sub}$; $h\neq 1$}\\
\\
\begin{array}{ll}(M_{i}P_{i,h+1}+M_{i,i-1}P_{i-1,h+1}+M_{i,i+1}P_{i+1,h+1})M_{h,h+1}^{T}\\ +(M_{i,i+1}P_{i+1,h}+ M_{i,i-1}P_{i-1,h})M_{h}^{T} \end{array} & \ \textrm{$1<i<n_{sub}$; $h=1$}
\\
\begin{array}{ll}(M_{i}P_{i,h-1}+M_{i,i-1}P_{i-1,h-1}+M_{i,i+1}P_{i+1,h-1})M_{h,h-1}^{T}\\ +(M_{i,i+1}P_{i+1,h} +M_{i,i-1}P_{i-1,h})M_{h}^{T} \end{array} & \ \textrm{$1<i<n_{sub}$; $h=n_{sub}$}
\\
\\
\begin{array}{ll}(M_{i}P_{i,h+1}+M_{i,i-1}P_{i-1,h+1}+M_{i,i+1}P_{i+1,h+1})M_{h,h+1}^{T} +(M_{i,i+1}P_{i+1,h}\\+ M_{i,i-1}P_{i-1,h})M_{h}^{T}+(M_{i}P_{i,h-1}+M_{i,i-1}P_{i-1,h-1}+M_{i,i+1}P_{i+1,h-1})M_{h,h-1}^{T} \end{array} & \ \textrm{otherwise}
\end{array}\right.;
\end{equation*}
and 
\begin{equation*}
 P_{\Omega_i \leftrightarrow \Omega_j}=\left\{\begin{array}{ll}
 \begin{array}{ll}
 M_{i,i+1}P_{i+1,i}M_{i}^T +  M_{i}P_{i,i+1}M_{i,i+1}^{T} +M_{i,i+1}P_{i+1}M_{i,i+1}^{T}\end{array} & \ \textrm{if $i=1$}\\
 \\
  \begin{array}{ll}
M_{i,i-1}P_{i-1,i}M_{i}^T + M_{i}P_{i,i-1}M_{i,i-1}^{T}+  M_{i,i-1}P_{i-1}M_{i,i-1}^{T} \end{array}& \ \textrm{if $i=n_{sub}$}\\
 \\
  \begin{array}{ll}
 (M_{i,i+1}P_{i+1,i}+M_{i,i-1}P_{i-1,i})M_{i}^T + M_{i}P_{i,i-1}M_{i,i-1}^{T}+ M_{i}P_{i,i+1}M_{i,i+1}^{T}\\ +(M_{i,i+1}P_{i+1,i-1}+M_{i,i-1}P_{i-1})M_{i,i-1}^{T}+(M_{i,i+1}P_{i+1}+M_{i,i-1}P_{i-1,i+1})M_{i,i+1}^{T} \end{array} & \ \textrm{otherwise}
 
 \end{array}\right.;
\end{equation*}

\noindent are the  matrices keeping track of contributions of  adjacent domains $\Omega_i$ and $\Omega_h$ to  overlapping region;  \noindent update of DD-KF gains:
\begin{equation*}
    K_i=\sum_{j=1\\ j\neq i}^{n_{sub}}(P_{i}H_{k+1}^{T}|_{I_{i}}+P_{i,j}H_{k+1}^{T}|_{I_{j}})\cdot F;
\end{equation*}
where 
\begin{equation*}
    F=(\sum_{i=1}^{n_{sub}} H_{k+1}|_{I_i}P_{i}H_{k+1}|_{I_{i}}^T +R_{1,\ldots,n_{sub}}+R_{k+1})^{-1}
\end{equation*}
and

\begin{equation*}
    R_{1,\ldots,n_{sub}}=\sum_{i=1}^{n_{sub}}(\sum_{j=1\\j\neq i}^{n_{sub}}H_{k+1}|_{I_j}P_{j,i})H_{k+1}|_{I_{i}}^T;
\end{equation*}
update of local estimate of matrices:
\begin{equation*}
    P_{i}=(I_{i}-K_i  H_{k+1}|_{I_i})P_{i}-\sum_{j=1\\ j \neq i}^{n_{sub}}(K_{i}H_{k+1}|_{I_{j}}P_{j,i});
\end{equation*}
 update of local estimate of  matrices between $e_i\in \mathbb{R}^{n_i}$ and $e_h\in \mathbb{R}^{n_h}$.
\begin{equation*}
    P_{i,h}=(I_{i}-K_i\cdot H_{k+1}|_{I_i}) P_{i,h}-\sum_{j=1\\ j\neq i}^{n_{sub}}(K_{i}H_{k+1}|_{I_{j}}P_{j,h}), \quad  h=1,\ldots,n_{sub} \quad h\neq i;
\end{equation*}
finally, we get to  the update of local estimates:
\begin{equation}\label{DD_KF_G_estimates}
    \widehat{x}_{i,k+1}={x}_{i,k+1}+K_{i}\cdot (y_{k+1}-\sum_{j=1 }^{n_{sub}}H_{k+1}|_{I_{j}}x_{j,k}).
\end{equation}
In computation of matrix $C_{\Omega_i \leftrightarrow \Omega_h}$, we refer to $P_{i}$ instead of $P_{i,i}$.\\

\noindent   DD-KF algorithm  is described in Table \ref{DD_KF} below. 

\begin{table}{}
\caption{ DD-KF procedure}.
\label{DD_KF}
\textbf{procedure}  DD-KF(in:$\,H_{0},\ldots,H_{lmax},Q_0,\ldots,Q_{lmax-1}$,$R_0,\ldots,R_{lmax},y_0,\ldots,$\\ $y_{lmax}$,$\mu$, $\delta$, out:$\,\widehat{x}_{i}$)\\
\textbf{Set}  index $j$ of $I_{j}$, i.e. the set adjacent to  $I_i$ \\
\ \ \ \ \ \ \ \ \textbf{if} (\textit{floor}$(i/2)\times 2==i$) \textbf{then}  $j:=i-1$  \%  $i$ is even\\
\ \ \ \ \ \ \ \ \ \ \ \textbf{else} $j:=i+1$  \%  $i$ is odd\\
\ \ \ \ \ \ \ \  \textbf{end} \\
\textbf{Call} DD-Setup (in:$\,y_0,H_0|_{I_i},H_0|_{I_j},R_0$,out:$\,\widehat{x}_{i,0},P_{i,0}$)\\
\textbf{Call} Local-KF (in:$\,y_{l}$, $H_{l-1}|_{I_{i}},H_{k}|_{I_{i}},B_{l-1}|_{I_{i}},R_{l},P_{i,0}$,$\mu$, $\delta$, $\widehat{x}_{i,l-1},\,$ out:$\,\widehat{x}_{i,l}^{n}$)\\
\textbf{Set} $\widehat{x}_{i}:=\widehat{x}_{i,l}^{n}$ \% DD-KF  estimate in $I_i$\\
\textbf{endprocedure}\\[1cm]

\textbf{procedure} DD-Setup(in: $\,y_0,H_0|_{I_i},H_0|_{I_j},R_0$, out: $\,\widehat{x}_{i,0},P_{i,0}$)\\
\textbf{Set up} reduced matrices: $H_{i,0}:= H_0|_{I_i}$, $H_{j,0}:= H_0|_{I_j}$ \\
\textbf{Compute} predicted covariance matrix $P_{i,0}= (H_{i,0}^{T}R_{0}H_{i,0})^{-1}$ \\
\textbf{Compute} $P_{H_{j,0}}$ \\
\textbf{Compute}  Kalman estimate: $\widehat{x}_{i,0}= (H_{i,0}^{T}P_{H_{j,0}}H_{i,0})^{-1}H_{i,0}^{T}P_{H_{j,0}}y_{0}$ \\
\textbf{endprocedure}\\[1cm]
\textbf{procedure}  Local-KF(in:$\,y_{k}$, $H_{l-1}|_{I_{i}},H_{l}|_{I_{i}},B_{l-1}|_{I_{i}},R_{l},P_{i,0}$,$\mu$, $\delta$, $\widehat{x}_{i,l-1}$; out:$\,\widehat{x}_{i,l}^{n}$)\\
\ \ \ \textbf{for} $l=1,\,lmax$ \%loop over  KF steps \\
\ \ \ \ \ \ $n:=0$, $\widehat{x}_{i,l}^{n+1}:=0$\\
\ \ \ \ \ \ \ \ \ \ \ \ \ \ \textbf{repeat}\\
\ \ \ \ \ \ \ \ \ \ \ \ \ \  \ \ \ $n:=n+1$\\
\ \ \ \ \ \ \ \ \ \ \ \ \ \  \ \ \ \textbf{Set up}  predicted covariance matrix $P_{i,l}$ \\
\ \ \ \ \ \ \ \ \ \ \ \ \ \  \ \ \ \textbf{Compute} Kalman gains $K_{i,l}$ \\
\ \ \ \ \ \ \ \ \ \ \ \ \ \  \ \ \ \textbf{Update}  covariance matrix $P_{i,l}$ \\
\ \ \ \ \ \ \ \ \ \ \ \ \ \  \ \ \ \textbf{Compute} matrices $K_{i,j}$ \\
\ \ \ \ \ \ \ \ \ \ \ \ \ \  \ \ \ \textbf{Send and Receive} boundary conditions among adjacent sets\\
\ \ \ \ \ \ \ \ \ \ \ \ \ \  \ \ \ \textbf{Exchange} of data  among $S_{I_{i}\leftrightarrow I_j}(\widehat{x}_{i,l}^{n})$ \\
\ \ \ \ \ \ \ \ \ \ \ \ \ \  \ \ \ \textbf{Compute} Kalman estimate at step $n+1$\\ \ \ \ \ \ \ \ \ \ \ \ \ \ \  \ \ \ $\widehat{x}_{i,l}^{n+1}=\widehat{x}_{i,l-1}+\textbf{K}_{i,l}\left[({y}_{l}-{H}_{l}|_{I_{j}}\widehat{x}_{j,l}^{n})-H_{l}|_{I_{i}}\widehat{x}_{i,l-1}\right]+\mathcal{S}_{I_{i}\leftrightarrow I_{j}}(\widehat{x}_{j,l}^{n})$ \\
\ \ \ \ \ \ \ \ \ \ \ \ \ \  \ \ \ \ \ \ \ \ \ \ \textbf{if}($s\neq 0$) \textbf{then} \% decomposition with overlap  \\
\ \ \ \ \ \ \ \ \ \ \ \ \ \  \ \ \ \ \ \ \ \ \ \ \ \ \ \textbf{Set up}  of  the extensions on ${I}_{i}$ of $\widehat{x}_{i,l}^{n+1}$ given  on $I_{i,j}$: $EO_{I_{i}}(\widehat{x}_{i,l}^{n}|_{{I}_{i,j}})$, $EO_{I_{i}}(\widehat{x}_{j,l}^{n}|_{{I}_{i,j}})$\\
\ \ \ \ \ \ \ \ \ \ \ \ \ \  \ \ \ \ \ \ \ \ \ \ \ \ \  \textbf{Exchange} data on the overlap set $I_{i,j}$:  $\nabla \mathcal{O}_{i,j}(EO_{I_{i}}(\widehat{x}_{i,l}^{n}|_{I_{i,j}}),EO_{I_{i}}(\widehat{x}_{j,l}^{n}|_{I_{i,j}}))$ \\
\ \ \ \ \ \ \ \ \ \ \ \ \ \  \ \ \ \ \ \ \ \ \ \ \ \ \ \textbf{Update} Kalman estimate at step $n+1$:\\ \ \ \ \ \ \ \ \ \ \ \ \ \ \  \ \ \ \ \ \ \ \ \ \ \ \ \ $\widehat{x}_{i,k}^{n+1}\leftarrow \widehat{x}_{i,l}^{n+1}+{\mu}\cdot  P_{i,k} \nabla \mathcal{O}_{i,j}(EO_{I_{i}}(\widehat{x}_{i,l}^{n}|_{I_{i,j}}),EO_{I_{i}}(\widehat{x}_{j,l}^{n}|_{I_{i,l}}))$ \\
\ \ \ \ \ \ \ \ \ \ \ \ \ \  \ \ \ \ \ \ \ \ \ \ \textbf{endif}\\
\ \ \ \ \ \ \ \ \ \ \ \ \ \ \textbf{until} ($\| \widehat{x}_{i,l}^{n+1}-\widehat{x}_{i,l}^{n} \| <TOL$)\\
\textbf{endfor} \% end of the loop over KF steps\\
\textbf{end procedure}\\ 
\end{table}

\section{Reliability assessment}

\noindent In this section we assess  reliability of the proposed method. In particular,  Theorem \ref{propeq}  proves the consistence of DD-KF, i.e.  state estimates $\widehat{x}_{1,k+1}\in \mathbb{R}^{n_{1}}$, $\widehat{x}_{2,k+1}\in \mathbb{R}^{n_{2}}$, as in (\ref{stimastato_DD}), are equal to  reductions of KF estimate $\widehat{x}_{k+1}\in \mathbb{R}^{n}$ to  $I_{1}$ and  $I_{2}$, respectively.  The same result can be obtained for  $n_{sub}>2$  spatial subdomains considering DD-KF estimates defined in (\ref{DD_KF_G_estimates}) and  proceeding as in proof of Theorem \ref{propeq}. In the following we let  $\Delta_{j}$ be fixed, then for simplicity of notation, we refer to $\widehat{x}_{i,k}^{\Delta_{j}}$ omitting superscript $\Delta_{j}$.

\begin{theorem}\label{propeq}
 
 Let $\widehat{x}_{k+1}\in \mathbb{R}^{n}$, $\forall k=\bar{s}_{j-1},\ldots,\bar{s}_{j-1}+s_{j}-1$ be  KF estimate in (\ref{stimapredetta1}) and $\widehat{x}_{1,k+1}\in \mathbb{R}^{n_{1}}$, $\widehat{x}_{2,k+1}\in \mathbb{R}^{n_{2}}$ be the DD-KF estimates in (\ref{stimastato_DD}). Then it holds that
\begin{equation}\label{equiv1}
\widehat{x}_{1,k+1}\equiv \widehat{x}_{k+1}|_{I_{1}}, \quad \widehat{x}_{2,k+1}\equiv \widehat{x}_{k+1}|_{I_{2}}. 
\end{equation}
\end{theorem}
\noindent \underline{\bf Proof}.
{\footnotesize 
For $k=\bar{s}_{j-1},\ldots,\bar{s}_{j-1}+s_{j}-1$, we prove that  DD-KF gains $K_{1}\in \mathbb{R}^{n_{1}\times  m\cdot d}$, $K_{2}\in \mathbb{R}^{n_{2}\times  m\cdot d}$ in (\ref{guadagnodikalman_DD}) are
\begin{equation}\label{equivalenza}
K_{1}\equiv K_{k+1}|_{I_{1}} \quad K_{2}\equiv K_{k+1}|_{I_{2}},
\end{equation}
where $K_{k+1}\in \mathbb{R}^{n\times m\cdot d}$ is the KF gain in (\ref{guadagnodikalman}). We first consider DD-KF without overlapping. KF gain can be written as follows
{\footnotesize
\begin{equation}\label{KFgain}
\begin{array}{ll}
\begin{array}{ll}
K_{k+1}&\equiv P_{k+1}H_{k+1}^{T}\cdot (H_{k+1}P_{k+1}H_{k+1}^{T}+R_{k+1})^{-1}\\ &=\left[\begin{array}{ll}
P_{k+1}|_{I_{1}\times I_{1}} \ P_{k+1}|_{I_1\times I_2}\\ P_{k+1}|_{I_2\times I_1}\ P_{k+1}|_{I_{2}\times I_{2}} 
\end{array}\right]\left[\begin{array}{ll} H_{k+1}|_{I_{1}}^{T}\\
H_{k+1}|_{I_2}^{T}\end{array}\right]
\\
&\cdot\left(\left[\begin{array}{ll}
H_{k+1}|_{I_{1}}\ H_{k+1}|_{I_{2}}
\end{array}\right]\left[\begin{array}{ll}
P_{k+1}|_{I_{1}\times I_{1}} \ P_{k+1}|_{I_{1}\times I_{2}}\\
 P_{k+1}|_{I_{2}\times I_{1}}\ P_{k+1}|_{I_{2}\times I_{2}} 
\end{array}\right]\left[\begin{array}{ll}
H_{k+1}|_{I_{1}}^{T}\\ H_{k+1}|_{I_{2}}^{T}
\end{array}\right]+R_{k+1}\right)^{-1}
\end{array},
\end{array}
\end{equation}
}
\noindent where $ P_{k+1}\in \mathbb{R}^{n\times n}$ is the predicted  covariance matrix, which is  defined in (\ref{predictedmatrix}).
We define the matrix
\begin{equation}
R_{1,2}:=(H_{k+1}|_{I_2}P_{k+1}|_{I_2\times I_{1}}H_{k+1}|_{I_1}^{T}+H_{k+1}|_{I_{1}}P_{k+1}|_{I_{1}\times I_{2}}H_{k+1}|_{I_2}^{T});
\end{equation}
and we obtain that 
\begin{equation}\label{res_kalmangains}
\begin{array}{ll}
\begin{array}{ll}
K_{k+1}|_{I_{1}}&=(P_{k+1}|_{I_{1}\times I_{1}}H_{k+1}|_{I_1}^{T}+P_{k+1}|_{I_1\times I_{2}}H_{k+1}|_{I_2}^{T})\\
&\cdot(H_{k+1}|_{I_1}P_{k+1}|_{I_{1}\times I_{1}}H_{k+1}|_{I_2}^{T}+H_{k+1}|_{I_2}P_{k+1}|_{I_2\times I_{2}}H_{k+1}|_{I_2}^{T}+R_{1,2}+R_{k+1})^{-1}
\end{array}
\\
\begin{array}{ll}
K_{k+1}|_{I_{2}}&=(P_{k+1}|_{I_{2}\times I_{2}}H_{k+1}|_{I_{2}}^{T}+P_{k+1}|_{I_{2}\times I_{1}}H_{k+1}|_{I_{1}}^{T})\\
&\cdot(H_{k+1}|_{I_{1}}P_{k+1}|_{I_{1}\times I_{1}}H_{k+1}|_{I_{2}}^{T}+H_{k+1}|_{I_{2}}P_{k+1}|_{I_{2}\times I_{2}}H_{k+1}|_{I_{2}}^{T}+R_{1,2}+R_{k+1})^{-1} 
\end{array}
\end{array}.
\end{equation}
We similarly obtain that 
\begin{equation}
\begin{array}{ll}
P_{k+1}|_{I_{1}\times I_{1}}\equiv P_{1} \quad P_{k+1}|_{I_{1}\times I_{2}}\equiv P_{1,2} \\

P_{k+1}|_{I_{2}\times I_{2}}\equiv P_{2} \quad P_{k+1}|_{I_{2}\times I_{1}}\equiv P_{2,1} 
\end{array}
\end{equation}
where $P_{1}$, $P_{2}$ are defined in (\ref{P1-P2}) and $P_{1,2}$, $P_{1,2}$ in (\ref{P12-P21}).
From (\ref{res_kalmangains}) and (\ref{guadagnodikalman_DD}) we obtain the equivalence in (\ref{equiv1}). We consider the predicted estimate $x_{k+1}$ in (\ref{stimapredetta}) so, KF estimate $\widehat{x}_{k+1}\in \mathbb{R}^{n}$ in (\ref{stimapredetta1}) can be written as follows
\begin{equation}
\left[\begin{array}{ll}
\widehat{x}_{k+1}|_{I_{1}}\\
\widehat{x}_{k+1}|_{I_{2}}
\end{array}\right]=\left[\begin{array}{ll}x_{k+1}|_{I_{1}}\\
x_{k+1}|_{I_{2}}
\end{array}\right]
+\left[\begin{array}{ll} K_{1}\\K_{2}\end{array}\right]\left( y_{k+1}-\left[\begin{array}{ll}
H_{k+1}|_{I_{1}}\ 
H_{k+1}|_{I_{2}}
\end{array}\right]\left[\begin{array}{ll}
{x}_{k}|_{I_{1}}\\
{x}_{k}|_{I_{2}}
\end{array}\right]\right).
\end{equation}
It is simply to prove that 
$$x_{k+1}|_{I_{1}}\equiv x_{1,k+1} \quad x_{k+1}|_{I_{2}}\equiv x_{2,k+1}, $$ where $x_{1,k+1}\in \mathbb{R}^{n_{1}}$ and $x_{2,k+1}\in \mathbb{R}^{n_{2}}$ are the predicted estimates in (\ref{stimastato_DD}), so we get the thesis in (\ref{equiv1}).\\[.5cm]
 In case of decomposition of $\Omega$ with overlap,  DD-KF estimates in (\ref{stimastato_DD})
can be written as follows:
\footnotesize
\begin{equation}\label{equiv_overlap}
\begin{array}{ll}
\widehat{x}_{i,k+1}&=x_{i,k+1}+K_{i}\left[y_{k+1}-(H_{k+1}|_{\tilde{I}_{1}}x_{1,k}|_{\tilde{I}_{1}} +\frac{1}{2}H_{k+1}|_{{I}_{1,2}}x_{1,k}|_{\tilde{I}_{1,2}}+\frac{1}{2}H_{k+1}|_{{I}_{1,2}}x_{2,k}|_{\tilde{I}_{1,2}} +H_{k+1}|_{\tilde{I}_{2}}x_{2,k})\right], \ i=1,2\\
\end{array}.
\end{equation}

\noindent Since 
\footnotesize
$$x_{k+1}|_{I_{1,2}}= x_{1,k+1}|_{I_{1,2}}=x_{2,k+1}|_{I_{1,2}}$$
(\ref{equiv_overlap}) becomes 
\footnotesize
\begin{equation}\label{DD-KF_overlap}
\begin{array}{ll}
\widehat{x}_{i,k+1}&=x_{i,k+1}+K_{i}\left[y_{k+1}-(H_{k+1}|_{\tilde{I}_{1}}x_{1,k}|_{\tilde{I}_{1}} +H_{k+1}|_{{I}_{1,2}}x_{k}|_{I_{1,2}}+H_{k+1}|_{\tilde{I}_{2}}x_{2,k})\right], \ i=1,2 \\
\end{array}.
\end{equation}
Noting that:
\begin{equation*}\label{dec_H}
    H_{k+1}=\left[H_{k+1}|_{\tilde{I}_1} \ H_{k+1}|_{{I}_{1,2}} \ H_{k+1}|_{\tilde{I}_2} \right],
\end{equation*}
\noindent  (\ref{equiv_overlap}) becomes:
\footnotesize
\begin{equation}\label{KF_o}
   \widehat{x}_{k+1}|_{I_i}=x_{k+1}|_{I_i}+K_{i}\left[y_{k+1}-(H_{k+1}|_{\tilde{I}_{1}}x_{k}|_{\tilde{I}_{1}} +H_{k+1}|_{{I}_{1,2}}x_{k}|_{I_{1,2}}+H_{k+1}|_{\tilde{I}_{2}}x_{k}|_{I_2})\right] 
\end{equation}

\noindent where $I_1$, $I_{1,2}$, $I_{2}$ and $\tilde{I}_{1}$, $\tilde{I}_{2}$ are defined in (\ref{set_indici}) and (\ref{set_indici_tilde}), respectively. Moreover,  equivalence in (\ref{equivalenza}) can be   similarly  obtained. Supposing (\ref{equiv1}) is true for  $k$, from (\ref{DD-KF_overlap}) and (\ref{KF_o}), we obtain  (\ref{equiv1}) for $k+1$. 
\\[.5cm]
    
}

\normalsize

\section{Experimental Results}

We apply DD-KF method to the initial boundary problem of SWEs. The discrete model is obtained  by using  Lax-Wendroff scheme \cite{LeVeque}. As DD-KF is intended to give the same solution of KF experiments are exclusively aimed to validate  DD-KF w.r.t. KF.  Reliability is assessed using maximum error and RMSE. Details of the whole derivation DD-KF on SWE are reported in the APPENDIX (Section \S \ref{section6}). \\

\noindent \textbf{KF configuration}. We consider the following experimental scenario:
\begin{itemize}
\item $\Omega=(0,L)$:   where $L=1$;
\item $\Delta=[0,T]$:   where $T=1.5$;
\item $p_v=2$: number of physical variables;
\item $n=500$ and $b_c=2$: numbers of elements of $\Omega$ and $\partial \Omega$, respectively; 
\item $n_{x}:=n+b_c=502$ and $nt=53$: numbers of elements of $\bar{\Omega}$ and $\Delta$, respectively;
\item $\Delta x=\frac{1}{500}$ and $\Delta t$: step size of $\Omega$ and $\Delta$, respectively, where $\Delta t$ varies to satisfy the stability condition in (\ref{S-Cbis}) of  Lax-Wendroff method;
\item $D_{n}({\Omega})=\{x_{i}\}_{i=0,\ldots,n-1}$: discretization of $\Omega$ where $x_{i}=i \cdot \Delta x$;
\item $D_{nt}({\Delta})=\{t_{k}\}_{k=0,\ldots,nt-1}$: discretization of $\Delta$ where $t_{k}=t_{k-1}+\Delta t$;
\item $\textbf{x}(t,x)\in \mathbb{R}^{2}$: state estimate with $t\in \Delta$ and $x\in \Omega$;
\item $m=14$:  number of observations at step $k=0,1,\ldots,nt-1$; 
\item $v_{k}:=10^{-2}\cdot \bar{v}_k\in \mathbb{R}^{m}$: observations errors, with $\bar{v}_k$ a random vector drawn from the standard normal distribution, for $k=0,1,\ldots,nt-1$;
\item $y_{k}:=x(t_{k},x_j)+v_{k}\in \mathbb{R}^{m}$: observations vector for $j=1,\ldots, m$ at  step $k=0,1,\ldots,nt-1$; observations are obtained by adding observation errors to the full solution of the SWE's without using domain decomposition while using exact initial and boundary conditions. 
\item $H_{k}\in \mathbb{R}^{m\times n}$: piecewise linear interpolation operator whose coefficients are computed using the
points of the domain nearest to  observation values;
\item $\sigma_{m}^{2}=5.0\times 10^{-1} $, $\sigma_{0}^{2}= 3.50\times 10^{-1}$: model and  observational error variances;
\item $Q\equiv Q_{k}=\sigma_{m}^{2} C$: covariance matrix of the model error at  step $k=0,1,\ldots,nt-1$, where $C\in \mathbb{R}^{n\times n}$ denotes  Gaussian correlation structure of  model errors in (\ref{matC});
\item $R\equiv R_{k}=\sigma_{0}^{2} I_{m,m}\in \mathbb{R}^{m\times m}$: covariance matrix of the errors of the observations at step $k=0,1,\ldots,nt-1$.
\end{itemize}

\noindent  \textbf{DD-KF configuration.} We consider the following experimental scenario:
\begin{itemize}
\item $s=2,4,\ldots,200$  size of the spatial overlap;
\item $n_{1}=250+s/2$, $I_{1}=\{1,\ldots,250\}$ with $|I_{1}|=n_{1}$;
\item $n_{2}=250+s/2$, $I_{1}=\{n_{1}-s+1,\ldots,500\}$ with $|I_{2}|=n_{2}$;
\item ${s}_{1,2}=2,3,\ldots,50$  size of the temporal overlap;
\item $\bar{s}_{0}:=0$, ${s}_{0,1}:=0$, $s_{1}=25+s_{1,2}/2$, $\bar{s}_{1}:=s_{1}-{s}_{1,2}$ and $s_{2}=28+s_{1,2}/2$;
\item $C:=\{c_{i,j}\}_{i,j=1,\ldots,n}\in \mathbb{R}^{n\times n}$: Gaussian correlation structure of  model error where
 \begin{equation}\label{matC}
c_{i,j}=\rho^{|i-j|^{2}}, \quad \rho=exp\left(\frac{-\Delta x^{2}
}{2L^{2}}\right), \quad |i-j|<n/2 \quad  \begin{array}{ll}\textrm{for}\ i=1,\ldots,n_{1},\\
j=1,\ldots,n_{1}-s \ \textrm{and}\\ \textrm{for} \ i,j=n_{1}-s+1,\ldots,n\end{array}.
\end{equation}
 According to \cite{JSC} here we assume  the model error to be gaussian . As explained in \cite{JSC},  a research collaboration between us and CMCC (Centro Euro Mediterraneo per i Cambiamenti
Climatici) give us the opportunity to use the software called OceanVar.
OceanVar is used in
Italy to combine observational data (Sea level anomaly, sea-surface temperatures, etc.) with
backgrounds produced by computational models of ocean currents for the Mediterranean Sea
(namely, the NEMO framework \cite{NEMO}). OceanVAR assumes gaussian model errors.
\end{itemize}

\noindent \textbf{Reliability Metrics}.
\noindent For $j=1,2$ and $k=\bar{s}_{j-1}+1,\ldots,\bar{s}_{j-1}+s_{j}$ and fixed the size of time overlap $s_{1,2}=1$, we compute $\widehat{x}_{k}[1]$, i.e.  KF estimate  of the wave height $h$ on $\Omega$. In order to quantify the difference between  KF estimate and  DD-KF estimates on $\Delta$ w.r.t. the size of the spatial overlap, e.g. the parameter  $s\in \mathbb{N}$, we use: $$error_{s}^{\Omega\times \Delta}:=\max(error_{s}^{\Omega\times \Delta_{1}},error_{s}^{\Omega\times \Delta_{2}})$$ where $$error_{s}^{\Omega\times \Delta_{j}}=\max_{k=\bar{s}_{j-1}+1,\ldots,\bar{s}_{j-1}+s_{j}}(||\widehat{x}_{k}[1]-\widehat{x}_{k}[1]^{\Omega\times \Delta_{j}}||)\,.$$ \\
 In the same way, we fix $s$  and compute the error between  KF estimate and  DD-KF estimates on $\Delta$ w.r.t. the size of the  overlap in time domain, e.g. the parameter  $s_{1,2}\in \mathbb{N}$:  $$error_{s_{1,2}}^{\Omega\times \Delta_{j}}:=\max_{k=\bar{s}_{j-1}+1,\ldots,\bar{s}_{j-1}+s_{j}}(||\widehat{x}_{k}[1]-\widehat{x}_{k}[1]^{\Omega\times \Delta_{j}}||)\,.$$\\

\noindent  Reliability of  DD-KF estimates is measured in terms of the Root Mean Square Error (RMSE) for $ k=0,1,\ldots,nt-1$, which is computed as
\begin{equation}
\begin{array}{ll}
RMSE_{k}^{\Omega_{1}\cup \Omega_{2}}=\sqrt{\frac{\sum_{i=1}^{n}({x}_{k}[1](i)-\widehat{x}_{k}^{\Omega\times \Delta}[1](i))^{2}}{n}} \\\quad
RMSE_{k}^{\Omega}=\sqrt{\sum_{i=1}^{n}\frac{({x}_{k}[1](i)-\widehat{x}_{k}[1](i))^{2}}{n}}.
\end{array}
\end{equation}
\noindent  Figure \ref{fig_err} shows $error_{s}^{\Omega\times \Delta}$ versus $s$,  obtained within the machine precision ($10^{-15}$). Also,  $error_{s_{1,2}}^{\Omega\times \Delta_{j}}$ where $s_{1,2}=200$ is again obtained  within the maximum attainable accuracy in double precision (i.e. in our case,  $10^{-15}$), with $j=1,2$ as shown in Figure \ref{fig_err2}.  As  expected,  the accuracy does not depend on the size of the overlap because,  DD-KF   is  a direct method using  estimates provided by KF as initial and boundary values of  the local dynamic model. In particular, in \ref{fig_err2} (b) when the size of temporal overlap reaches $50$ we observe a  relative  increment of the error of one unit, very significant with respect to the overall magnitude of the error.  This effect is due to the increasing impact of  round off errors on the accuracy of the solution. As expected, in a DD method, the extra work  performed on the overlapped region with an  increasing  size can be seen as the effect of a preconditioner on  overlapping region which can overestimate the solution \cite{Chan}.\\
\noindent  Figure \ref{fig_RMSE} shows that the $RMSE_{k}^{\Omega_{1}\cup \Omega_{2}}$ and $RMSE_{k}^{\Omega}$ decrease, in particular as we expected, $RMSE_{k}^{\Omega_{1}\cup \Omega_{2}}$ and $RMSE_{k}^{\Omega}$ coincide during the whole assimilation window.\\

\noindent \textbf{Qualitative analysis}. A qualitative analysis of DD-KF estimate $\widehat{x}_{k+1,}[1]^{\Omega \times \Delta_{j}}$ at time $t_{25}\in \Delta_{1}$ and $t_{53}\in \Delta_{2}$, i.e. for $k+1=25,53$ and $j=1,2$, respectively is shown in Figure \ref{fig_data}. In Figure \ref{fig_data} (a), we note that $\widehat{x}_{25}[1]^{\Omega \times \Delta_{1}}$ moves from the trajectory of the model state $x[1]_{25}$ to DD-KF estimate position closer to the observation. At the second observation there is a significantly smaller alteration of the trajectory towards the observation. At the fifth observation, it is very close to the observation so we would not expect much effect from the assimilation of this observation. As the model evolves in
time it is clear to see that  observations have a diminishing effect on the correction
of the forecast state estimate, as we can note in Figure \ref{fig_data} (b). In particular, for different choices of  $\sigma_{m}^{2}$ and $\sigma_{0}^{2}$ (model and  observation error variances),  trajectory of DD-KF estimates $\widehat{x}_{k+1}[1]^{\Omega \times \Delta_{j}}$  changes, for $j=1,2$. Figure \ref{fig_R_Q} (a) shows that for $\sigma_{m}^{2}=0$,  DD-KF method gives full confidence to the model, indeed, trajectory of  $\widehat{x}_{25}[1]^{\Omega \times \Delta_{1}}$ coincides with the model state $x[1]_{25}$, otherwise considering  $\sigma_{0}^{2}=10^{-5}$,  DD-KF method gives more confidence to observations,  as shown in Figure \ref{fig_R_Q} (b).\\ 

\noindent \textbf{Capability of DD-KF to deal with the presence of different observation errors}.   To point out the capability of DD-KF to deal with the presence of different observation errors, in Figure \ref{fig_var} we  reports results obtained considering two different values of observations errors  in $\Omega_1$ and $\Omega_2$, i.e.  $v_k=[v_k^1 v_k^2]$, respectively;    in particular, in  Figure \ref{fig_var} (a) we set  $v_k^1=10^{-15} \times \bar{v}_k^1$ in $\Omega_1$ and $v_k^2=2  \bar{v}_k^2$ in $\Omega_2$; in  Figure \ref{fig_var} (b) we set   $v_k^1=  \bar{v}_k^1$ in $\Omega_1$ and $v_k^2=10^{-15} \times \bar{v}_k^2$ in $\Omega_2$, where $\bar{v}_k^{1}$ and $\bar{v}_k^{2}$ are random vectors drawn from the standard normal distribution and $k=10$. Results shown in  Figure \ref{fig_var} (a) confirm the effect of the variance observations $\sigma_{2,0}^{2}=  4.28 \times 10^{0}$ on  DD-KF estimate $\widehat{x}_{k}[1]^{\Omega \times \Delta}$  in  $\Omega_2$ ; similarly, it occurs  in $\Omega_1$ when $\sigma_{1,0}^{2} = 1.04 \times 10^{0}$, as results in Figure \ref{fig_var} (b) show. As a consequence DD-KF method gives more confidence to model state than to the observations in $\Omega_2$ (see Figure  Figure \ref{fig_var} (a)) or in $\Omega_1$ (see Figure \ref{fig_var} (b)), respectively, depending on the variance observations. \\

\begin{figure}[h!]
\centering
{\includegraphics[width=.6\textwidth]{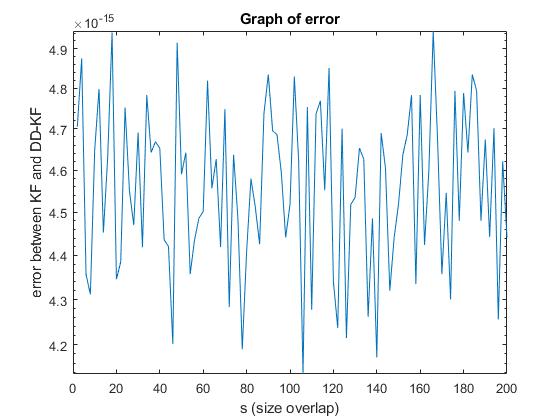}}
\caption{Graph of $error_{s}^{\Omega\times \Delta}$ versus  the spatial overlap $s$ while  time overlap is fixed to $s_{1,2}=1$.}
\label{fig_err}
\end{figure}

\begin{figure}[h!]
\centering
\subfloat[][\emph{Behaviour of $error_{s_{1,2}}^{\Omega\times \Delta_{1}}$ versus the size of time overlap $s_{1,2}$.}]
{\includegraphics[width=.68\textwidth]{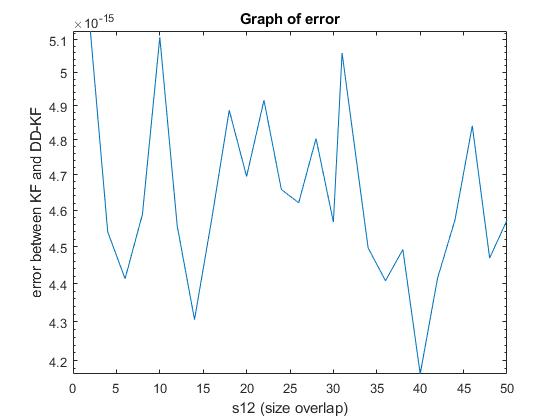}}\quad
\subfloat[][\emph{Behaviour of $error_{s_{1,2}}^{\Omega\times \Delta_{2}}$ versus  the size of time overlap $s_{1,2}$.}]
{\includegraphics[width=.68\textwidth]{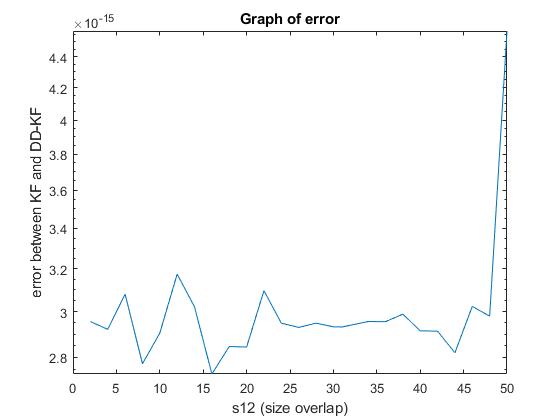}}
\caption{Behaviour  of  $error_{s_{1,2}}^{\Omega\times \Delta_{1}}$ and $error_{s_{1,2}}^{\Omega\times \Delta_{2}}$ versus  the size of time overlap $s_{1,2}$ while  the size of spatial overlap is $s=200$.}
\label{fig_err2}
\end{figure}

\begin{figure}[h!]
\centering
{\includegraphics[width=.68\textwidth]{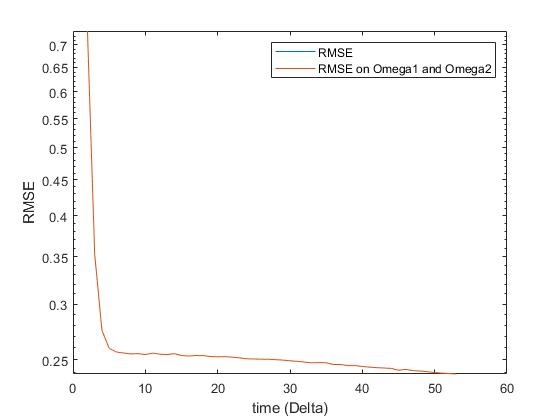}}\quad
\caption{Behaviour of $RMSE_{k}^{\Omega_{1}\cup \Omega_{2}}$ and $RMSE_{k}^{\Omega}$ versus time. }
\label{fig_RMSE}
\end{figure}

\begin{figure}[h!]
\centering
\subfloat[][\emph{SWE solution $x[1]_{25}\in \mathbb{R}^{n_{x}}$  and the DD-KF estimate $\widehat{x}_{25}^{\Omega\times \Delta_{1}}[1]\in \mathbb{R}^{n_{x}}$  at  $t_{25}$ on $\Omega$.}]
{\includegraphics[width=.68\textwidth]{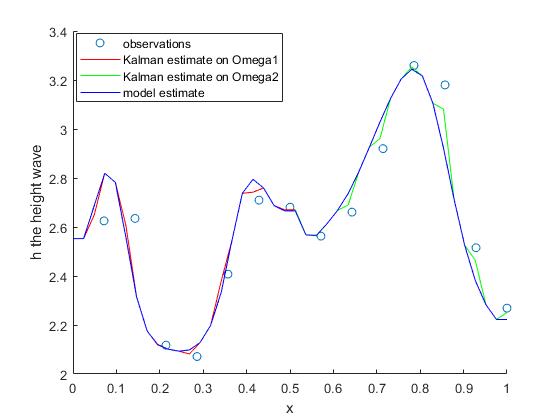}} \quad
\subfloat[][\emph{SWE solution  $x[1]_{53}\in \mathbb{R}^{n_{x}}$ (model estimate) and DD-KF estimate  $\widehat{x}_{53}^{\Omega\times \Delta_{2}}[1]\in \mathbb{R}^{n_{x}}$ at $t_{53}$ on $\Omega$.}]
{\includegraphics[width=.68\textwidth]{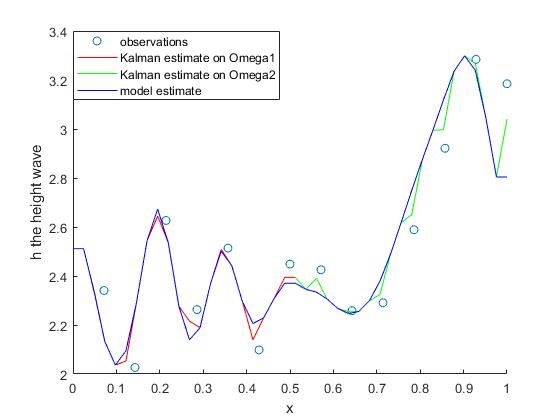}}
\caption{SWE solution and its DD-KF estimate.}
\label{fig_data}
\end{figure}

\begin{figure}[h!]
\centering
 \subfloat[][\emph{SWE solution $x[1]_{53}\in \mathbb{R}^{n_{x}}$  (model estimate) and  DD-KF estimates  $\widehat{x}_{53}^{\Omega\times \Delta_{1}}[1]\in \mathbb{R}^{n_{x}}$ at  $t_{53}$ on $\Omega$ (Kalman estimate 1 and Kalman estimate 2) by considering $\sigma_{m}^{2}=0$. As expected, Kalman estimates and model estimate overlap.}]
{\includegraphics[width=.68\textwidth]{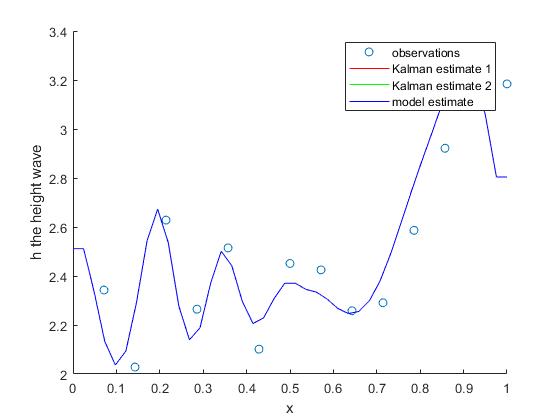}}\quad
\subfloat[][\emph{SWE solution   $x[1]_{99}\in \mathbb{R}^{n_{x}}$ (model estimate)  and DD-KF estimates $\widehat{x}_{25}^{\Omega\times \Delta_{1}}[1]\in \mathbb{R}^{n_{x}}$  at $t_{53}$ on $\Omega_1$  and $\Omega_{2}$ (Kalman estimate 1 and Kalman estimate 2) by considering the error variance $\sigma_{0}^{2}=10^{-5}$.}]
{\includegraphics[width=.68\textwidth]{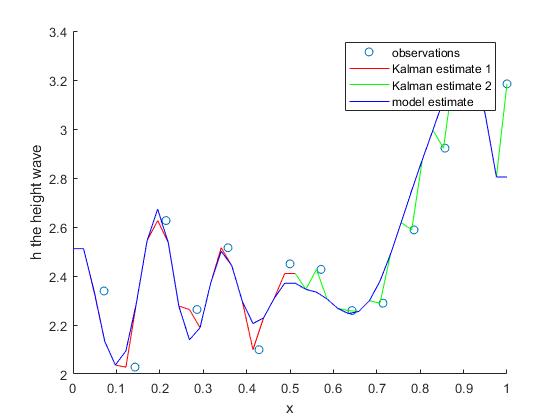}}
\caption{SWE solution and DD-KF estimates  for different choices of  $\sigma_{m}^{2}$ and $\sigma_{0}^{2}$.}
 \label{fig_R_Q}
\end{figure}

\begin{figure}[h!]
\centering
 \subfloat[][\emph{SWE solution $x[1]_{10}\in \mathbb{R}^{n_{x}}$ (model estimate)   and DD-KF estimates $\widehat{x}_{10}^{\Omega\times \Delta} [1]\in \mathbb{R}^{n_{x}}$  at $t_{10}$ on $\Omega$  (Kalman estimate 1 and Kalman estimate 2) by considering different  errors and variances observations; i.e. $v_{10}^{1}=10^{-15}\cdot \bar{v}^{1}_{10}$ and $\sigma_{1,0}^{2} = 6.67 \times 10^{-1}$ in $\Omega_1$ and $v_{10}^{2}=2\cdot \bar{v}_{10}^{2}$ and $\sigma_{2,0}^{2} = 4.28 \times 10^{0}$ in $\Omega_2$. $\bar{v}_{10}^{1}$ and $\bar{v}_{10}^{2}$ are random vectors drawn from the standard normal distribution.}]
{\includegraphics[width=.68\textwidth]{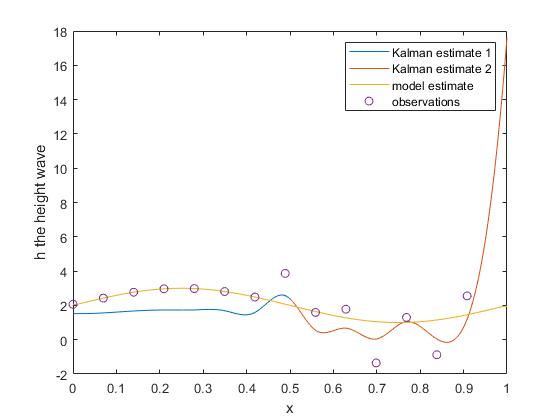}}\quad
\subfloat[][\emph{SWE solution $x[1]_{10}\in \mathbb{R}^{n_{x}}$ (model estimate)  and DD-KF estimates $\widehat{x}_{10}^{\Omega\times \Delta} [1]\in \mathbb{R}^{n_{x}}$ (Kalman estimate 1 and Kalman estimate 2) at $t_{10}$ on $\Omega$ by considering different  errors and variances observations; i.e. $v_{10}^{1}=1\cdot \bar{v}^{1}_{10}$ and $\sigma_{1,0}^{2} = 1.04 \times 10^{0}$ in $\Omega_1$ and $v_{10}^{2}=10^{-15}\cdot \bar{v}_{10}^{2}$ and $\sigma_{2,0}^{2} = 2.22 \times 10^{-1}$ in $\Omega_2$.  $\bar{v}_{10}^{1}$ and $\bar{v}_{10}^{2}$ are random vectors drawn from  standard normal distribution.}]
{\includegraphics[width=.68\textwidth]{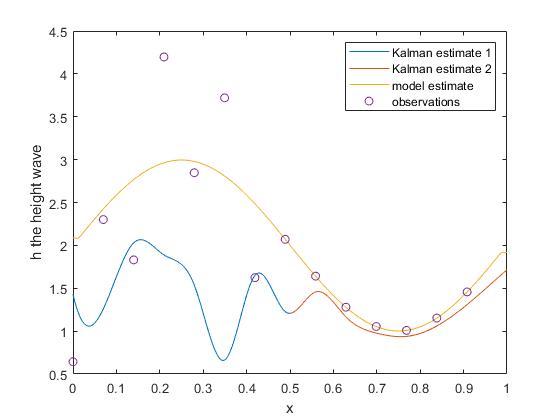}}
\caption{SWE solution and DD-KF estimates for different errors $v_{10}^{1}$, $v_{10}^{2}$ and variances observations $\sigma_{1,0}^{2}$ and $\sigma_{2,0}^{2}$  in  $\Omega_1$ and $\Omega_2$, respectively.   In order to point out the capability of DD-KF to deal with the presence of different observation errors,  we  show results obtained considering two different values of observations errors  in $\Omega_1$ and $\Omega_2$, i.e.  $v_k=[v_k^1 v_k^2]$, respectively.  Figure \ref{fig_var} (a) confirms the effect of the variance observations $\sigma_{2,0}^{2}=  4.28 \times 10^{0}$ on  DD-KF estimate $\widehat{x}_{k}[1]^{\Omega \times \Delta}$  in  $\Omega_2$; similarly, it occurs  in $\Omega_1$ when $\sigma_{1,0}^{2} = 1.04 \times 10^{0}$, as Figure \ref{fig_var} (b) shows. As a consequence DD-KF method gives more confidence to model state than to the observations in $\Omega_2$ (a) or in $\Omega_1$ (b), depending on the variance observations.}
 \label{fig_var}
\end{figure}

\section{Concluding remarks and future works}
We have presented the mathematical  framework which makes it possible  using KF  in  any application involving real time  state estimation. Borrowing Schwarz DD and PinT methods from PDEs, the framework  properly combines KF localization and model reduction in space and time, fully revisited  in a linear algebra settings. DD-KF framework makes it natural to switch from a full scale solver to a model order reduction solver for solution of subproblems for which no relevant low-dimensional reduced space should be constructed. Leveraging Schwarz and PinT  methods consistency constraints for PDEs-based models, the framework iteratively adjust local solutions by adding  the contribution of adjacent subdomains to the local filter, along  overlapping regions.
\noindent We derived and discussed main features of DD-KF framework using SWEs which are commonly used for monitoring and forecasting the water flow in rivers and open channels. Nevertheless, this framework has application on the plentiful  literature of PDE-based state estimation real-world problems. As DD-KF is intended to give the same solution of KF experiments are exclusively aimed to validate  DD-KF w.r.t. KF. Comparison against other KF methods which are approximations of KF, such as EnKF,  are out of focus of the present work.  Rather we note that the underlying idea of  DD-KF  could be applied to EnKF by changing the filter,  as it is performed in DD-KF. Specifically, in order to enforce the matching of local solutions on  adjacent  regions, local  problems should be  slightly modified by adding the smoothness-regularization constraint  to the correction phase on local solutions; such   term keeps track of contributions of adjacent domains to ensembles regions. The same modification should be done on the covariance matrices. \\
\noindent This work makes possible numerous  extensions. Among them, it makes it possible to apply
deep-learning techniques to develop consistency constraints which will ensure
that the solutions are physically meaningful even at the
boundary of the small domains in the output of the local models \cite{deeplearning4,Hahnel,deeplearning3}. Another possible extension  could be the employment of a dynamic load balancing scheme based
on  adaptive and dynamic redefining of initial decomposition \cite{arxivPPAM2019,DyDD}. Specifically, in order to optimally choose the domain decomposition configuration,  the partitioning into subdomains must satisfy certain conditions. First the computational
load assigned to subdomains must be equally distributed. Good quality partitioning also requires the volume of communication during
calculation to be kept at its minimum. In \cite{arxivPPAM2019,DyDD} the authors  employed a dynamic load balancing scheme based
on  adaptive and dynamic redefining of initial decomposition,   aimed to balance load between processors according to data
location.  In particular, the authors
focused on the introduction of a dynamic redefining of initial DD in order to deal with problems where the observations are non uniformly distributed and general sparse. This is a quite common issue in DA. More generally, one would consider DD framework in multi-physics and multi-resolution applications, by employing  adaptive and dynamic DD depending on the requirements of the application.   Overall, the employment of the framework in high performance computing state estimation systems such as on board  satellite navigation systems or autonomous vehicle tracking with GPS seems to be a fruitful research area. 
\section{Data Availability}
The Matlab code used for the numerical simulations is available from the corresponding author upon request.

\section{APPENDIX }\label{section6}

In this section, with the aim of allowing the replicability of experimental results, we describe the the whole detailed derivation of DD-KF method on  SWEs.

\subsection{Preliminaries on DD}

\noindent \textbf{Definition (reduction of matrices):}
Let $B=[B^{1} \ B^{2} \ \ldots \ B^{n}] \in \mathbb{R}^{m \times n}$ be a matrix with $m,n\ge 1$ and $B^{j}$ the $j-th$ column of $B$ and $I_{j}=\{1,\ldots,j\}$ and $I_{i,j}=\{i,\ldots,j\}$ for $i=1,\ldots,n-1$ and $j=2,\ldots,n$. The reduction of $B$ to the set $I_{j}$ is:
\begin{equation}
|_{I_{j}}: \ B\in \mathbb{R}^{m\times n} \rightarrow \ B|_{ I_{j}}=[B^{1} \ B^{2} \ \ldots \ B^{j}]\in \mathbb{R}^{m \times j}, \quad \textit{ $j=2,\ldots,n$},
\end{equation}
and to $I_{i,j}$
\begin{eqnarray}{l}
|_{ I_{i,j}}: \ B\in \mathbb{R}^{m\times n} \rightarrow \ B|_{ I_{i,j}}=[B^{i} \ B^{i+1} \ \ldots \ B^{j}]\in \mathbb{R}^{m \times j-i},  \\ 
\textit{ $i=1,\ldots,n-1$, $j=2,\ldots,n$}, 
\end{eqnarray}
where $B|_{ I_{j}}$ and $B|_{ I_{i,j}}$ denote the reductions of  $B$ to $I_{j}$ and $I_{i,j}$, respectively.\\[.2cm]
\textbf{Definition (Extension  of vectors):}
Let $w=[w_{t} \ w_{t+1} \ \ldots \ w_{n}]^{T} \in \mathbb{R}^{s}$ be a vector with $t\ge 1$, $n>1$, $s=n-t$. The extension  of $w$  to  $I_{r}$ where $r>n$ is:
\begin{equation}
EO_{I_{r}}: \ w\in \mathbb{R}^{s} \rightarrow \ EO_{I_{r}}(w)=[\bar{w}_{1} \ \bar{w}_{2} \ \ldots \ \bar{w}_{r}]^{T} \in \mathbb{R}^{r},
\end{equation}
where, for $i=1,\ldots,r$, it is: 
\begin{equation}
\bar{w}_{i} =\left\{ \begin{array}{ll} 
w_{i} & \textit{if $t\le i\le n$}\\
0 & \textit{if $i> n$ and $i<t$ \quad .}\\
\end{array} \right.
\end{equation}

\subsection{DD-KF vs KF Performance Analysis}
 \noindent  In order to  quantify the benefit of the decomposition, i.e., how much  the time complexity of the KF algorithm  may be reduced  with respect to the DD--KF based algorithm we use the so called scale-up factor introduced in \cite{CAI}.\\
\begin{definition} (\textbf{Scale-up factor})
Let $N_p \in \mathbb{N}$ be the problem size and $N_{sub}\ge 2$ the number of subdomains. Let $T_{KF}(N_p)$ and  $T_{DD-KF}(N_{sub})$ be  time complexity of  KF algorithm   and DD-KF algorithm  respectively. We introduce the scale-up factor:
\begin{equation}\label{factor}
Sc^f_{N_{sub},1}(N_p)=\frac{T_{KF}(N_p)}{N_{sub}\cdot T_{DD-KF}(N_{sub})}.
\end{equation}
\end{definition}

\noindent Following result allows us to analyze the behaviour of the scale up factor.

\begin{proposition}

Let  $r \equiv N_p/N_{sub}$. Then it holds that
\begin{equation*}
Sc^f_{N_{sub},1}(N_p)=\alpha(N_p,N_{sub})N_{sub}^{2}
\end{equation*} 
where
\begin{equation}
\alpha(N_p,N_{sub})=\frac{a_{3}+a_{2}\frac{1}{N_p}+a_{1}\frac{1}{N_p^{2}}+a_{0}\frac{1}{N_p^{3}}}{b_{3}+b_{2}\frac{1}{r}+b_{1}\frac{1}{r^{2}}+b_{0}\frac{1}{r^{3}}}.
\end{equation}
\end{proposition}

\begin{proof}
We let
$T_{KF}(N_p)=p_{1}(N_p)$ and  $T_{DD-KF}(r)=p_{2}(r)$ where $p_1$ and $p_2$ are polynomials of degree less than or equal 3 i.e. $p_{1},p_{2}\in \Pi_{3}$:
\begin{equation}
\begin{array}{ll}
T_{KF}(N_p)=a_{3}N_p^{3}+a_{2}N_p^{2}+a_{1}N_p+a_{0}\\
T_{DD-KF}(r)=b_{3}r^{3}+b_{2}r^{2}+b_{1}r+b_{0}
\end{array}.
\end{equation}
It holds
\begin{equation}
\begin{array}{ll}
Sc^f_{N_{sub},1}(m,N_p)&=\frac{a_{3}N_p^{3}+a_{2}N_p^{2}+a_{1}N_p+a_{0}}{N_{sub}(b_{3}r^{3}+b_{2}r^{2}+b_{1}r+b_{0})}\cdot r^{3}\cdot \frac{1}{r^{3}}=\frac{a_{3}N_{sub}^{3}+a_{2}\frac{n_{sub}^{3}}{N_p}+a_{1}\frac{N_{sub}^{3}}{N_p^{2}}+a_{0}\frac{N_{sub}^{3}}{N_p^{3}}}{N_{sub}(b_{3}+b_{2}\frac{1}{r}+b_{1}\frac{1}{r^{2}}+b_{0}\frac{1}{r^{3}})}\\
\\
&=\frac{N_{sub}^{3}\cdot(a_{3}+a_{2}\frac{1}{N_p}+a_{1}\frac{1}{N_p^{2}}+a_{0}\frac{1}{N_p^{3}})}{N_{sub}(b_{3}+b_{2}\frac{1}{r}+b_{1}\frac{1}{r^{2}}+b_{0}\frac{1}{r^{3}})}=\frac{a_{3}+a_{2}\frac{1}{N_p}+a_{1}\frac{1}{N_p^{2}}+a_{0}\frac{1}{N_p^{3}}}{b_{3}+b_{2}\frac{1}{r}+b_{1}\frac{1}{r^{2}}+b_{0}\frac{1}{r^{3}}}\cdot N_{sub}^{2}
\end{array}
\end{equation}
and 
\begin{equation*}
Sc^f_{N_{sub},1}(N_p)=\alpha(N_p,N_{sub})N_{sub}^{2}
\end{equation*} 
where
\begin{equation}
\alpha(N_p,N_{sub})=\frac{a_{3}+a_{2}\frac{1}{N_p}+a_{1}\frac{1}{N_p^{2}}+a_{0}\frac{1}{N_p^{3}}}{b_{3}+b_{2}\frac{1}{r}+b_{1}\frac{1}{r^{2}}+b_{0}\frac{1}{r^{3}}}.
\end{equation}
\end{proof}


\noindent Recalling that $r=\frac{N
_p}{N_{sub}}$ and $Sc^f_{N_{sub},1}(N_p)=\alpha(N_p,N_{sub})N_{sub}^{2}$, this result says that
the performance gain of DD-KF algorithm in terms of  reduction of time complexity, scales as the number of subdomains $N_{sub}$ squared, where the scaling factor depends on  $\alpha(N_p,N_{sub})$, that is,  on the time complexity of the algorithms.

\subsection{SWE setup}
Neglecting the Coriolis force and frictional forces and assuming unit width, SWEs are:
\begin{equation}\label{SWEsetup}
\left \{
\begin{array}{ll}
\frac{\partial h}{\partial t}+\frac{\partial vh}{\partial {x}}=0\\
\frac{\partial vh}{\partial t}+\frac{\partial (v^{2}h+\frac{1}{2}gh^{2})}{\partial {x}}=0
\end{array}
\right . 
\end{equation}
where the independent variables ${x}$ and $t$ make up the spatial dimension and time. The dependent
variables are $h$, which is the height with respect to the surface and $v$, i.e. 
the horizontal velocity, while $g$ is the gravitational acceleration.\\
In order to write  SWEs in a compact form, we introduce the vectors
\begin{equation}
\textbf{x}:=\left[\begin{array}{ll}
h\\
vh
\end{array}\right]\quad 
f(\textbf{x}):=\left[\begin{array}{ll}
vh\\
v^{2}h+\frac{1}{2}gh^{2}
\end{array}\right],
\end{equation}
the SWEs can be rewritten as follows
\begin{equation}\label{80}
\frac{\partial \textbf{x}}{\partial t}+\frac{\partial f(\textbf{x)}}{\partial {x}}=0.
\end{equation}
The initial boundary problem for the SWEs is 
\begin{equation}
\left\{\begin{array}{lllll}
\textbf{x}(t+\Delta t,x)=\mathcal{M}_{t,t+\Delta t}(\textbf{x}(t,x)) \quad \forall t, t+\Delta t \in [0,1.5]\\
\textbf{x}(0,x):=\left[\begin{array}{ll}h(0,x)\\
h(0,x)v(0,x)
\end{array}\right] =\left[\begin{array}{ll} 2+\sin(2\pi x)\\
0 \end{array} \right] \quad \forall x \in \Omega \end{array}\right.
\end{equation}
with the following reflective boundary conditions on $v$ and free boundary conditions on $h$
\begin{equation}\label{boundary_cond}
\textbf{x}(t,0):=\left[\begin{array}{ll}h(t,x_{1})\\
-h(t,x_{1})v(t,x_{1})
\end{array}\right]\quad
\textbf{x}(t,x_{n_{x}-1}):=\left[\begin{array}{ll}h(t,x_{n_{x}-2})\\
-h(t,x_{n_{x}-2})v(t,x_{n_{x}-2})
\end{array}\right] \quad \forall t\in \Delta
\end{equation} 
where
\begin{equation}\label{83}
\mathcal{M}_{t,t+\Delta t}(\textbf{x}(t,x)):=\textbf{x}(t,x)-\int_{t}^{t+\Delta t} \frac{\partial f(\textbf{x})}{\partial x} ds.
\end{equation}
We note that  $\mathcal{M}_{t,t+\Delta t}(\textbf{x}(t,x))$ should include Coriolis forces and also frictional forces, if they are present in the SWEs; in particular, as these quantities  show off as the right hand side of (\ref{80}), they will be also included in the right hand side of (\ref{83}) as additive terms under  the integral. \\ 
The state of the system at each time $t_{k+1}\in \Delta$, $k=0,1,\ldots,nt-2$ is:
\begin{equation}
\textbf{x}(t_{k+1})\equiv \textbf{x}_{k+1}:=\left[\begin{array}{lll}
x[1]_{k+1}\\
x[2]_{k+1}  
\end{array}\right]\in \mathbb{R}^{2(n_{x}-2)}
\end{equation}
where
\begin{equation}
 \begin{array}{lll}
  x[1]_{k+1}&\:= \{x[1](t_{k+1},x_{i})\}_{i=1,\ldots,n_{x}-1} &:=\{h(t_{k+1},x_{i})\}_{i=1,\ldots,n_{x}-1} \in \mathbb{R}^{(n_{x}-2)} \\ x[2]_{k+1}&:= \{x[2](t_{k+1},x_{i})\}_{i=1,\ldots,n_{x}-1}&:=\{v(t_{k+1},x_{i})h(t_{k+1},x_{i})\}_{i=1,\ldots,n_{x}-1} \in \mathbb{R}^{(n_{x}-2)} 
 \end{array}.
 \end{equation}
\noindent From  Lax-Wendroff scheme \cite{LeVeque}, we obtain the following discrete formulation of the SWEs
\begin{equation}\label{modello_swes}
\textbf{x}_{k+1}=M_{k,k+1}\textbf{x}_{k}+b_{k}+w_{k}
\end{equation}
where 
\begin{equation}
M_{k,k+1}=\left[\begin{array}{lllllllllllll}
M[1]_{k,k+1} \ \ O_{n_{x}-2}\\
M[2,1]_{k,k+1} \ \ M[2]_{k,k+1}
\end{array}\right] \in \mathbb{R}^{2(n_{x}-2)\times 2(n_{x}-2)} 
\end{equation}
with $O_{n_{x}-2}\in \mathbb{R}^{n_{x}-2\times n_{x}-2}$ the null matrix and $M[1]_{k,k+1}\in \mathbb{R}^{(n_{x}-2)\times (n_{x}-2)}$, $M[2,1]_{k,k+1}\in \mathbb{R}^{(n_{x}-2)\times (n_{x}-2)}$, $M[2]_{k,k+1}\in \mathbb{R}^{(n_{x}-2)\times (n_{x}-2)}$ the following tridiagonal matrices 
{\footnotesize
\begin{equation}\label{section7M}
M[1]_{k,k+1}=\left[\begin{array}{lllllllllllll}
\psi_{1}^{k} & \eta_{2}^{k} & & & \\
-\eta_{1}^{k}& \psi_{2}^{k}  &  \eta_{3}^{k} & &  \\
& \ddots & \ddots & \ddots   \\
& &-\eta_{n_{x}-4}^{k} & \psi_{n_{x}-3}^{k}  &\eta_{n_{x}-2}^{k}  \\
& &  &-\eta_{n_{x}-3 }^{k} & \psi_{n_{x}-2}^{k}
\end{array}\right]
\end{equation}
\begin{equation}
M[2,1]_{k,k+1}=\left[\begin{array}{lllllllllllll}
0 & \chi_{2}^{k}& & & \\
-\chi_{1}^{k}& 0 &  \chi_{3}^{k} & & \\
& \ddots & \ddots & \ddots   \\
& &-\chi_{n_{x}-4}^{k} & 0  &\chi_{n_{x}-2}^{k}  \\
& &  &-\chi_{n_{x}-3 }^{k}  & 0
\end{array}\right]
\end{equation}
\begin{equation}
M[2]_{k,k+1}=\left[\begin{array}{lllllllllllll}
\phi_{1}^{k} & \xi_{2}^{k}& & & \\
-\xi_{1}^{k}& \phi_{2}^{k} & \xi_{3}^{k} & & \\
& \ddots & \ddots & \ddots   \\
& &-\xi_{n_{x}-4}^{k} & \phi_{n_{x}-3}^{k}  &\xi_{n_{x}-2}^{k}  \\
& &  &-\xi_{n_{x}-3}^{k}  & \phi_{n_{x}-2}^{k}
\end{array}\right]
\end{equation}
and the vector $b\in \mathbb{R}^{2(n_{x}-2)}$
\begin{equation}
b_{k}:=\left[\begin{array}{ll}
b[1]_{k}\\
b[2]_{k}
\end{array}\right]\in \mathbb{R}^{2(n_{x}-2)}
\end{equation}
with 
\begin{equation}\label{b[1]}
b[1]_{k}:=\left[\begin{array}{ll}
\eta_{0}^{k}h_{0}^{k}\\
0\\
\vdots\\
0\\
\eta_{n_{x}-1}^{k}h_{n_{x}-1}^{k}
\end{array}\right]\in \mathbb{R}^{n_{x}-2} \quad 
b[2]_{k}:=\left[\begin{array}{ll}
\xi_{0}^{k}v_{0}^{k}h_{0}^{k}-\chi_{0}^{k}h_{0}^{k}\\
0\\
\vdots
\\
0
\\
\xi_{n_{x}-1}^{k}h_{n_{x}-1}^{k}+\xi_{n_{x}-1}^{k}v_{n_{x}-1}^{k}h_{n_{x}-1}^{k}
\end{array}\right]\in \mathbb{R}^{n_{x}-2} 
\end{equation}
}
\noindent where 
\begin{displaymath}
\begin{array}{lllllll}
\eta_{i}^{k}&:=\frac{\alpha}{2}[v_{i}^{k}+\alpha((v_{i}^{k})^{2}+\frac{1}{2}gh_{i}^{k})] & \quad  \psi_{i}^{k}&:=1-4\alpha^{2}((v_{i}^{k})^{2}+\frac{1}{2}gh_{i}^{k})\\
\xi_{i}^{k}&:=\frac{\alpha}{2}[\alpha + v_{i}^{k}] & \quad  \phi_{i}^{k}&:=\left(1+4\alpha^{2}\right)
\end{array}
\end{displaymath}
\begin{displaymath}
\chi_{i}^{k}:=\frac{1}{2}gh_{i}^{k}
\end{displaymath}
with $\alpha:=\frac{\Delta t}{2\Delta x}$ and $h_{i}^{k}:=h(t_{k},x_{i})$, $v_{i}^{k}:=v(t_{k},x_{i})$.\\
We note that the discrete model in (\ref{modello_swes}) can be rewritten as follows
\begin{equation}\label{1-2_model}
\begin{array}{ll}
x[1]_{k+1}={M}[1]_{k,k+1}x[1]_{k}+b[1]_{k}\\
x[2]_{k+1}={M}[2]_{k,k+1}x[2]_{k}+\tilde{b}[2]_{k}
\end{array}
\end{equation}
with 
\begin{equation}\label{tilde_b[2]}
\tilde{b}[2]_{k}:={b}[2]-M[2,1]_{k,k+1}x[1]_{k}.
\end{equation}
In particular, if we define 
\begin{equation}\label{S-C}
S:=\max \left(|v-\sqrt{gh}|,|v+\sqrt{gh}|\right),
\end{equation}
the stability condition of  Lax-Wendroff is
\begin{equation}\label{S-Cbis}
S\cdot \frac{\Delta t}{\Delta x} \le 1,
\end{equation}
for the condition (\ref{S-Cbis}) to be satisfied, at each iteration $k=0,1,\ldots,nt$ we choose: $\Delta t=0.8\cdot \frac{\Delta x}{S}$.
\begin{enumerate}
\item Decomposition Step.
\begin{itemize}
\item Decomposition of  $\Omega$ into two subdomains with overlap region:
\begin{equation}
\begin{array}{ll} 
\Omega_{1}=[0,x_{n_{1}}] \quad \textrm{with $D_{n_{1}+1}(\Omega_1)=\{x_{i}\}_{i=0,1,\ldots,n_1}$ }\\
\Omega_{2}=[x_{n_{1}-s+1},1] \quad  \textrm{with $D_{n_{2}+1}({\Omega}_{2})=\{x_{i}\}_{i=n_{1},\ldots,n-1}$}
\end{array}
\end{equation}
\item Decomposition of  $\Delta$ into two subsets with overlap region:
\begin{equation}
\begin{array}{ll}
\Delta_{1}=[0,t_{s_{1}-1}] \quad \textrm{with $D_{s_{1}}(\Delta_{1})=\{t_{k}\}_{k=0,\ldots,s_{1}-1}$}  \\
\Delta_{2}=[t_{s_{1}-s_{1,2}+1},1.5] \quad  \textrm{with $D_{s_{2}}(\Delta_{2})=\{t_{k}\}_{k=s_{1}-s_{1,2}+1,\ldots,nt-1}$}.
\end{array}
\end{equation}
\end{itemize}
\item Local initial conditions on  $\Delta_{1}$ and $\Delta_{2}$:
\begin{equation}\label{initial}
\begin{array}{ll}
x[1]_{0}^{\Delta_{1}}=h(0,x)=:x[1]_{0}^{\Delta_{0}}\\
x[2]_{0}^{\Delta_{1}}=h(0,x)v(0,x)=:x[2]_{0}^{\Delta_{0}}
\end{array}\quad
\begin{array}{ll}
x[1]_{{s_{1}-1}}^{\Delta_{2}}=x[1]_{{s_{1}-1}}^{\Delta_{1}}\\
x[2]_{{s_{1}-1}}^{\Delta_{2}}=x[2]_{{s_{1}-1}}^{\Delta_{1}}
\end{array}.
\end{equation}
 
\item Model decomposition in (\ref{1-2_model}).
\\
Decomposition of matrices $M[1]_{k+1}\in \mathbb{R}^{n_{x}-2}$ at each step $k=0,1,\ldots,nt-2$ 
{\footnotesize
\begin{equation}
M[1]_{k,k+1}^{1}:=M[1]_{k,k+1}|_{I_{1}\times I_{1}}=\left[\begin{array}{lllllllllllll}
\psi_{1}^{k} & \eta_{2}^{k} & & & \\
-\eta_{1}^{k}& \psi_{2}^{k} &  \eta_{3}^{k} & &  \\
& \ddots & \ddots & \ddots   \\
& &-\eta_{n_{1}-2}^{k} & \psi_{n_{1}-1}^{k}  &\eta_{n_{1}}^{k}  \\
& &  &-\eta_{n_{1}-1 }^{k} &  \psi_{n_{1}}^{k} 
\end{array}\right]\in \mathbb{R}^{n_{1}\times n_{1}}
\end{equation}
\begin{equation}
M[1]_{k,k+1}^{1,2}:=M[1]_{k,k+1}|_{I_{1}\times I_{2}}=\left[\begin{array}{lllllllllllll}
0 &\cdots  &\cdots &\cdots &0 &\cdots &0 \\
0& \cdots &\cdots  &\cdots &0 &\cdots & 0 \\
\vdots & &  & & & &\vdots \\
& &  & & \\
0&\cdots &0& \eta_{n_{1}+1}^{k} &0    & \cdots  &0  \\
\end{array}\right]\in \mathbb{R}^{n_{1}\times n_{2}}
\end{equation}
\begin{equation}
M[1]_{k,k+1}^{2,1}:=M[1]_{k,k+1}|_{I_{1}\times I_{2}}=\left[\begin{array}{lllllllllllll}
0 & \cdots&0 &0 &\eta_{n_{1}}^{k} &0 &\cdots &0 \\
0& \cdots &  0 &\cdots &0 &0 &\cdots &0 \\
\vdots& & & & \vdots &\vdots & & \vdots \\
& &  & & \\
0& \cdots &  0 &\cdots &0&0 &\cdots &0 \\
\end{array}\right]\in \mathbb{R}^{n_{2}\times n_{1}}
\end{equation}
\begin{equation}
M[1]_{k,k+1}^{2}:=M[1]_{k,k+1}|_{I_{2}\times I_{2}}=\left[\begin{array}{lllllllllllll}
\psi_{n_{1}}^{k} & \eta_{n_{1}+1}^{k} & & & \\
-\eta_{n_{1}}^{k}& \psi_{n_{1}+1}^{k} &  \eta_{n_{1}+2}^{k} & &  \\
& \ddots & \ddots & \ddots   \\
& &-\eta_{n_{x}-4}^{k} & \psi_{n_{x}-3}^{k} &\eta_{n_{x}-2}^{k}  \\
& &  &-\eta_{n_{1}-3 }^{k} & \psi_{n_{x}-2}^{k}
\end{array}\right]\in \mathbb{R}^{n_{2}\times n_{2}}.
\end{equation}
}
\noindent We  proceed in the same way to get the decomposition of  matrix $M[2]_{k,k+1} \in \mathbb{R}^{n_{x}-2\times n_{x}-2}$ in (\ref{1-2_model}) into $M[2]_{k,k+1}^{1} \in \mathbb{R}^{n_{1}\times n_{1}}$, $M[2]_{k,k+1}^{1,2}\in \mathbb{R}^{n_{1}\times n_{2}}$, $M[2]_{k,k+1}^{2,1}\in \mathbb{R}^{n_{2}\times n_{1}}$ and $M[2]_{k,k+1}^{2}\in \mathbb{R}^{n_{2}\times n_{2}}$.

\item Decomposition of  $H_{k+1}\in \mathbb{R}^{m\times (n_{x}-2)}$ into $H_{k+1}^{1}\in \mathbb{R}^{m\times n_{1}}$ and $H_{k+1}^{2}\in \mathbb{R}^{m\times n_{2}}$ as in (\ref{mat_H}).
\end{enumerate}

\subsection{DD-KF method.}
\begin{enumerate}
\item For $j=1,2$ and $k=\bar{s}_{j-1}+1,\ldots,\bar{s}_{j-1}+s_{j}$, we apply  DD-KF method on $\Delta_{1}$ and $\Delta_{2}$ by considering the following matrices:
\begin{equation}
\begin{array}{llllll}
M_{1}[i]&\equiv M[i]_{k,k+1}^{1} &\quad M_{1,2}[i]&\equiv M[i]_{k,k+1}^{1,2}\\
M_{2}[i]&\equiv M[i]_{k,k+1}^{2}& \quad M_{2,1}[i]&\equiv M[i]_{k,k+1}^{2,1}\\
P_{1}&=O_{n_{1}\times n_{1}}& \quad  P_{2}&=O_{n_{2}\times n_{2}}\\ P_{1,2}&=O_{n_{1}\times n_{2}} & \quad P_{2,1}&=O_{n_{2}\times n_{1}}
\end{array}.
\end{equation}
In particular, we note:
\begin{equation}
    M_{1,2}[i]\equiv \left[\begin{array}{lll} O & M_{1,3}[i]\\ 
    O& M_{2,3}[i]
    \end{array}\right]\quad  M_{2,1}[i]\equiv \left[\begin{array}{lll} \tilde{M}_{1,2}[i] & O\\ 
     \tilde{M}_{2,1}[i] & O
    \end{array}\right].
\end{equation}
\item Send and receive  boundary conditions from  adjacent domains and compute the vectors:
\begin{equation}\label{boundary}
b_{1,k}[i]=\left[\begin{array}{lll} M_{1,3}[i]\\ 
    M_{2,3}[i]
    \end{array}\right]\widehat{x}_{2,k}[i]|_{\Gamma_{1}} \quad
    b_{2,k}[i]=\left[\begin{array}{lll} \tilde{M}_{1,2}[i]\\ 
    \tilde{M}_{2,1}[i]
    \end{array}\right]\widehat{x}_{1,k}[i]|_{\Gamma_{2}}.
\end{equation}
\item For $j=1,2$ and $k=\bar{s}_{j-1}+1,\ldots,\bar{s}_{j-1}+s_{j}$, compute the predicted state estimates $x_{1,k+1}[i]\in \mathbb{R}^{n_{1}}$ and $x_{2,k+1}[i]\in \mathbb{R}^{n_{2}}$ for $ i=1,2$, as follows:
\begin{equation}
\begin{array}{ll}
x_{1,k+1}[i]=M_{1}[i]\widehat{x}[i]_{k}^{\Delta_{j-1}}|_{I_{1}}+\bar{b}[i]_{k}|_{I_{1}}+b_{1,k}\\
x_{2,k+1}[i]=M_{2}[i]\widehat{x}[i]_{\bar{s}_{j}}^{\Delta_{j-1}}|_{I_{2}}+\bar{b}[i]|_{I_{2}}+b_{2,k}
\end{array},
 \end{equation}
 with 
 \begin{equation}
 \bar{b}[i]_{k}=\left\{\begin{array}{lll}
 b[1]_{k} & \textrm{in (\ref{b[1]})} & \textrm{if $i=1$}\\
  \tilde{b}[2]_{k} &\textrm{in (\ref{tilde_b[2]})} & \textrm{if $i=2$} 
 \end{array}.\right.
 \end{equation}
\item For $j=1,2$ and $k=\bar{s}_{j-1},\ldots,\bar{s}_{j-1}+s_{j}-1$, compute DD-KF estimates as in (\ref{stimastato_DD}) on $\Delta_{j}$ i.e. $ \widehat{x}_{1,k+1}^{\Delta_{j}}[i] \equiv \widehat{x}[i]_{1,k+1}\in \mathbb{R}^{n_{1}}$ and $\widehat{x}_{2,k+1}^{\Delta_{j}}[i] \equiv \widehat{x}[i]_{2,k+1}\in \mathbb{R}^{n_{2}}$. In particular, in the time interval $\Delta_{j}$, by considering the boundary conditions in (\ref{boundary_cond}), DD-KF estimates on $\Omega_{1}$ and $\Omega_{2}$ are 
\begin{equation}
\begin{array}{ll}
\widehat{x}_{k+1}^{\Omega_{1}\times \Delta_{j}}[i] &:=\left[\begin{array}{ll}
{x}(t_{k+1},0)[i]\\
\widehat{x}_{1,k+1}^{\Delta_{j}}[i]\\
\widehat{x}_{2,k+1}^{\Delta_{j}}[i](1)
\end{array}\right]\in \mathbb{R}^{n_{1}+1} \\
\widehat{x}_{k+1}^{\Omega_{2}\times \Delta_{j}}[i] &:=\left[\begin{array}{ll}
\widehat{x}_{1,k+1}^{\Delta_{j}}[i](n_{1}-s)\\
\widehat{x}_{2,k+1}^{\Delta_{j}}[i]\\
{x}(t_{k+1},x_{n_{x}-1})[i]
\end{array}\right]\in \mathbb{R}^{n_{2}+1} 
\end{array}
\end{equation}
where $\widehat{x}_{r,k+1}^{\Delta_{j}}[i](1)$, $\widehat{x}_{r,k+1}^{\Delta_{j}}[i]( n_{1}-s)$ are the first and the $n_{1}-s$ components of $\widehat{x}_{r,k+1}^{\Delta_{j}}[i]$, $r=1,2$.
We refer to 
\begin{equation}
\widehat{x}_{k+1}[i]^{\Omega\times \Delta_{j}}:=\left[\begin{array}{ll}
\widehat{x}_{k+1}^{\Omega_{1}\times \Delta_{j}}[i]\\
\widehat{x}_{k+1}^{\Omega_{1,2}\times \Delta_{j}}[i]\\
\widehat{x}_{k+1}^{\Omega_{2} \times \Delta_{j}}[i]
\end{array}\right]\in \mathbb{R}^{n_{x}}
\end{equation}
as the estimate of the wave height and velocity (if $i=1,2$ respectively) obtained by applied the DD-KF method on $\Omega$, where on the spatial overlap $\Omega_{1,2}$, we have considered the arithmetic mean between DD-KF estimates i.e.
\begin{equation}
 \widehat{x}_{k+1}^{\Omega_{1,2}\times \Delta_{j}}[i]:=\left[\begin{array}{ll}
 \frac{\widehat{x}_{k+1}^{\Omega_{1}\times \Delta_{j}}[i]|_{I_{1,2}
}+ \widehat{x}_{k+1}^{\Omega_{2}\times \Delta_{j}}[i]|_{I_{1,2}
}}{2} \end{array}\right]
\end{equation}
with $I_{1,2}$ the index set defined in (\ref{set_indici}).
\end{enumerate}
\begin{table}
\label{DD-SWE}
\caption{ DD-KF to SWEs}
    \textbf{procedure DD-KF-SWE}\\
{\bf Domain Decomposition Step}: Decomposition of $\Omega$ and of $\Delta$\\
{\bf SWE Model setup} as in (\ref{SWEsetup}) -- (\ref{S-C})\\
{\bf Define} Local initial conditions on  $\Delta_{1}$ and $\Delta_{2}$ as in (\ref{initial})\\
{\bf Reduction of $M_{k,k+1}$} in (\ref{1-2_model})\\
{\bf Reduction of  $H_{l+1}\in \mathbb{R}^{n_{obs}\times (n_{x}-2)}$} as in (\ref{mat_H}).\\
{\bf for} $k=1,2$ \\
{\bf for } $l=\bar{s}_{k-1}+1,\ldots,\bar{s}_{k-1}+s_{k}$, as in (\ref{essesign}) \% apply  DD-KF method on $\Delta_{k}$ \\
 {\bf Send and receive}  boundary conditions from  adjacent domains as in (\ref{boundary} )\\{\bf Compute} 
\begin{equation}\nonumber
b_{1,l}[i]=\left[\begin{array}{lll} M_{1,3}[i]\\ 
    M_{2,3}[i]
    \end{array}\right]\widehat{x}_{2,l}[i]|_{\Gamma_{1}} \quad
    b_{2,l}[i]=\left[\begin{array}{lll} \tilde{M}_{1,2}[i]\\ 
    \tilde{M}_{2,1}[i]
    \end{array}\right]\widehat{x}_{1,l}[i]|_{\Gamma_{2}}.
\end{equation}
{\bf Compute} the predicted state estimates $x_{1,l+1}[i]\in \mathbb{R}^{n_{1}}$ and $x_{2,l+1}[i]\in \mathbb{R}^{n_{2}}$ 
\begin{equation}\nonumber
\begin{array}{ll}
x_{1,l+1}[i]=M_{1}[i]\widehat{x}[i]_{l}^{\Delta_{k-1}}|_{I_{1}}+\bar{b}[i]_{l}|_{I_{1}}+b_{1,l}\\
x_{2,l+1}[i]=M_{2}[i]\widehat{x}[i]_{\bar{s}_{k}}^{\Delta_{k-1}}|_{I_{2}}+\bar{b}[i]|_{I_{2}}+b_{2,l}
\end{array},
 \end{equation}
 with 
 \begin{equation}\nonumber
 \bar{b}[i]_{l}=\left\{\begin{array}{lll}
 b[1]_{l} & \textrm{in (\ref{b[1]})} & \textrm{if $i=1$}\\
  \tilde{b}[2]_{l} &\textrm{in (\ref{tilde_b[2]})} & \textrm{if $i=2$} 
 \end{array}.\right.
 \end{equation}
{\bf Compute DD-KF estimates} as  in (\ref{stimastato_DD}) on $\Omega_i \times\Delta_{k}$ \\  
\begin{equation}\nonumber
\begin{array}{ll}
\widehat{x}_{l+1}^{\Omega_{1}\times \Delta_{k}}[i] &:=\left[\begin{array}{ll}
{x}(t_{l+1},0)[i]\\
\widehat{x}_{1,l+1}^{\Delta_{k}}[i]\\
\widehat{x}_{2,l+1}^{\Delta_{k}}[i](1)
\end{array}\right]\in \mathbb{R}^{n_{1}+1} \\
\widehat{x}_{l+1}^{\Omega_{2}\times \Delta_{k}}[i] &:=\left[\begin{array}{ll}
\widehat{x}_{1,l+1}^{\Delta_{k}}[i](n_{1}-\delta/2)\\
\widehat{x}_{2,l+1}^{\Delta_{k}}[i]\\
{x}(t_{l+1},x_{n_{x}-1})[i]
\end{array}\right]\in \mathbb{R}^{n_{2}+1} 
\end{array}
\end{equation}
{\bf Compute DD-KF estimate on $\Omega \times \Delta_k$ } 
\begin{equation}\nonumber
\widehat{x}_{l+1}[i]^{\Omega\times \Delta_{k}}:=\left[\begin{array}{ll}
\widehat{x}_{l+1}^{\Omega_{1}\times \Delta_{k}}[i]\\
\widehat{x}_{l+1}^{\Omega_{1,2}\times \Delta_{k}}[i]\\
\widehat{x}_{l+1}^{\Omega_{2} \times \Delta_{k}}[i]
\end{array}\right]\in \mathbb{R}^{n_{x}}
\end{equation}
and  
\begin{equation}\nonumber
 \widehat{x}_{l+1}^{\Omega_{1,2}\times \Delta_{k}}[i]:=\left[\begin{array}{ll}
 \frac{\widehat{x}_{l+1}^{\Omega_{1}\times \Delta_{k}}[i]|_{I_{1,2}
}+ \widehat{x}_{l+1}^{\Omega_{2}\times \Delta_{k}}[i]|_{I_{1,2}
}}{2} \end{array}\right]
\end{equation}
with $I_{1,2}$  defined in (\ref{set_indici}).\\
{\bf endprocedure}
\end{table}
\footnotesize
\begin{table}[]
    \centering
    \begin{tabular}{|l|l|} \hline
   $ \{t_{k}\}_{k=0,1,\ldots,r+1}$,  $t_{k}=k\Delta t$, $\Delta t= \frac{T}{r+1}$ & points of $\Delta= [0,T]$\\
   $x_{k}\equiv x(t_{k})$ & state  at $t_{k}$         \\
   $M_{k,k+1}\in \mathbb{R}^{N\times N}$& discretization of the tangent\\ &linear operator  of $\mathcal{M}_{t,t+\Delta t}( x(t))$      \\
   $y_{k}\equiv y(t_{k})\in \mathbb{R}^{m\cdot d}$&  observations vector\\
   $H_{k+1}\in \mathbb{R}^{m\cdot d \times N}$&  discretization of tangent\\ & linear approximation of $\mathcal{H}_{t_{k+1}}$\\
   $M=\left[\begin{array}{c|c}
M_{1} & M_{1,2}\\
\hline
M_{2,1} & M_{2}
 \end{array} \right]$ eqs. (\ref{deco_M})-(\ref{set_indici_tilde})
 & block decomposition of $M= M_{k,k+1}$ \\
 $ H_{k+1}=\left[H_{k+1}^{1} \ H_{k+1}^{2}\right]$ eqs. (\ref{deco_M})-(\ref{set_indici_tilde}) & block decomposition of $H_{k+1}$\\
$x_{1,\bar{s}_{j-1}}^{\Delta_{j}}=x_{1,\bar{s}_{j-1}}^{\Delta_{j-1}}\quad  x_{2,\bar{s}_{j-1}}^{\Delta_{j}}=x_{2,\bar{s}_{j-1}}^{\Delta_{j-1}}$ & DD-KF initial conditions of $M_{k,k+1}$\\
$ \Gamma_{1}:=\partial \Omega_{1}\cap \Omega_{2} \quad, \quad   \Gamma_{2}:=\partial \Omega_{2}\cap \Omega_{1}$ &  boundary of$\Omega_1$ and $\Omega_2$\\
$x_{1,k}^{\Delta_{j}}|_{\partial \Omega_{1} \setminus \Omega_{2}}={x}_{k}^{\partial \Omega_{1} \setminus \Omega_{2}} $ &  boundary conditions on $\Omega_1$\\
$x_{1,k}^{\Delta_{j}}|_{\Gamma_{1}} = x_{2,k}|_{\Gamma_{1}}$ & boundary conditions  on $\Omega_1$\\
$x_{2,k}^{\Delta_{j}}|_{\partial \Omega_{2} \setminus  \Omega_{1}}={x}_{k}^{\partial \Omega_{2} \setminus \Omega_{1}}$ &boundary conditions  on $\Omega_2$ \\
$ x_{2,k}^{\Delta_{j}}|_{\Gamma_{2}}  = x_{1,k}|_{\Gamma_{2}}$ & boundary conditions  on $\Omega_2$\\
   $Q_{k}\in \mathbb{R}^{N\times N}$ & covariance matrices \\ &of the errors on the model\\
    $R_{k}\in \mathbb{R}^{m\cdot d \times m\cdot d}$& covariance matrices \\& of the errors  on the observations\\
    $D_{n}(\Omega)=\{{x}_{j}\}_{j=1,\ldots,n}$ & the mesh partitioning of $\Omega$\\
   $ \Omega_{1,2}:=\Omega_{1}\cap \Omega_{2}$ &  the space overlap region\\
  $ D_{s}(\Omega_{1,2}):=D_{n_{1}}(\Omega_{1})\cap D_{n_{2}}(\Omega_{2})=\{x_{i}\}_{i\in I_{1,2}}$& 
 the discretization of $\Omega_{1,2}$\\
$I_{1}=\{1,\ldots,n_{1}\}$ & the index set of $D_{n_{1}}(\Omega_{1})$\\
$I_{2}=\{n_{1}-s+1,\ldots,n\}$ & the index set of $D_{n_{2}}(\Omega_{2})$\\
$ I_{1,2}=\{n_{1}-s+1,\ldots,n_{1}\}$ & the index set of $D_{n_{1}}(\Omega_{1})$\\
$D_{s_{j}}(\Delta_{j})=\{t_{k}\}_{k=\bar{s}_{j-1},\ldots,\bar{s}_{j-1}+s_{j}}$& the discretization of $\Delta_j$\\
$ \bar{s}_{j-1}:=\sum_{l=1}^{j-1}(s_{l}-s_{l-1,l}), \quad  s_{j-1,j}\in \mathbb{N}_{0}$ &
the number of elements in common\\ &  between  subsets $\Delta_{j-1}$ and $\Delta_{j}$\\
$Q_{k}=diag(Q_{1,k},Q_{2,k})\in \mathbb{R}^{n\times n}$ &  covariance matrix\\ 
$P_{1}=M_{1}P_{1}M_{1}^{T}+P_{\Omega_{1}\leftrightarrow \Omega_{2}}+Q_{1,k}$ & local covariance matrix on $\Omega_1$\\
$P_{2}=M_{2}P_{2}M_{2}^{T}+P_{\Omega_{2}\leftrightarrow \Omega_{1}} +Q_{2,k}$&local covariance matrix on $\Omega_2$\\
$P_{i,j}:=Cov(e_{i,k+1},e_{j,k+1})$ & local covariance matrix on the overlap \\
$P_{2,1}=M_{2}P_{2,1}M_{1}^{T}+C_{\Omega_{1}\leftrightarrow \Omega_{2}}^{T}$&local covariance matrix on $\Omega_{1,2}$\\
$P_{1,2}=M_{1}P_{1,2}M_{2}^{T}+C_{\Omega_{1}\leftrightarrow \Omega_{2}}$ &local covariance matrix on $\Omega_{1,2}$\\
$C_{\Omega_{1}\leftrightarrow \Omega_{2}} =M_{1}P_{1}M_{2,1}^{T}+M_{1,2}P_{2,1}M_{2,1}^{T}+M_{1,2}P_{2}M_{2}^{T}$ &  contribution of \\&$\Omega_{1}$ and $\Omega_{2}$ to  overlapping region\\ 
$K_{1}=(P_{1}H_{k+1}|_{I_{1}}^{T}+P_{1,2}H_{k+1}|_{I_{2}}^{T})\cdot F$ & DD-KF gain\\
$K_{2}=(P_{2}H_{k+1}|_{I_{2}}^{T}+P_{2,1}H_{k+1}|_{I_{1}}^{T})\cdot F
$& DD-KF gain\\
 $F=(H_{k+1}|_{I_{1}}P_{1}H_{k+1}|_{I_{1}}^{T}+H_{k+1}|_{I_{2}}P_{2}H_{k+1}|_{I_{2}}^{T} +$ &\\$+R_{1,2}+R_{k+1})^{-1}$&\\
$ R_{1,2}=(H_{k+1}|_{I_{2}}P_{2,1}H_{k+1}|_{I_{1}}^{T}+H_{k+1}|_{I_{1}}P_{1,2}H_{k+1}|_{I_{2}}^{T})$&
 contribution of \\&$\Omega_{1}$ and $\Omega_{2}$ to  overlapping region\\
$  P_{1}=(I-K_{1}H_{k+1}|_{I_{1}})P_{1}-K_{1}H_{k+1}|_{I_{2}}P_{2,1}$ & DD-KF update of covariance matrix\\
$P_{2}=(I-K_{2}H_{k+1}|_{I_{2}})P_{2}-K_{2}H_{k+1}|_{I_{1}}P_{1,2}$&DD-KF update of covariance matrix\\
$\widehat{x}_{k+1}^{\Delta_{j}}=M_{k,k+1}\widehat{x}_{k}^{\Delta_{j}}+b_{k}+w_{k}, \quad \forall k=\bar{s}_{j-1},\ldots,\bar{s}_{j-1}+s_{j}-1$ & local model  state estimate\\
$x_{1,k+1}=M_{1}\widehat{x}_{1,k}+b_{k}|_{I_{1}}+b_{1,k}+w_k|_{I_{1}}$ & local DD-KF state estimate on $\Omega_1$\\
$x_{2,k+1}=M_{2}\widehat{x}_{2,k}+b_{k}|_{I_{2}}+b_{2,k}+w_k|_{I_{2}}$ & local  state estimate on  $\Omega_2$\\
$\widehat{x}_{1,k+1}=x_{1,k+1}+K_{1}\left[y_{k+1}-(H_{k+1}|_{I_{1}}x_{1,k}+H_{k+1}|_{I_{2}}x_{2,k})\right] $ &DD-KF estimate on $\Omega_1$ \\
$\widehat{x}_{2,k+1}=x_{2,k+1}+K_{2}\left[y_{k+1}-(H_{k+1}|_{I_{1}}x_{1,k}+H_{k+1}|_{I_{2}}x_{2,k})\right]$&DD-KF estimate on $\Omega_2$ \\ \hline

    \end{tabular}
    \caption{  Table collecting DD-KF notations}
    \label{tab:my_label}
\end{table}
\section{Declarations}
The authors confirm that the research described in this  work has not received any funds.\\
\noindent The authors confirm that there are not any conflicts of interest.\\
\noindent The authors confirm that data and code can be available at request.


\clearpage

\begin{thebibliography}{99}
\bibitem{Anderson} J. L. Anderson, An Ensemble Adjustment Kalman Filter for data assimiliation.
Mon. Wea. Rev., 129, 2001, pp. 2884-2903.
\bibitem{JPP} R. Arcucci, L.  D'Amore, L. Carracciuolo, G. Scotti, G. Laccetti. A Decomposition of the Tikhonov Regularization Functional oriented to exploit hybrid multilevel parallelism. International  Journal of Parallel Programming, vol. 45, 2017, pp. 1214-1235.  ISSN: 0885-7458, doi: 10.1007/s10766-016-0460-3


\bibitem{WCEAS} R. Arcucci, L.  D'Amore,  L. Carracciuolo, On the problem-decomposition of scalable 4D-Var Data Assimilation model, Proceedings of the 2015 International Conference on High Performance Computing and Simulation, HPCS 2015
2 September 2015,  pp. 589-594, 13th International Conference on High Performance Computing and Simulation, HPCS 2015; Amsterdam; Netherlands; 20 July 2015 through 24 July 2015.

\bibitem{JCP} R. Arcucci, L. D'Amore, J. Pistoia, R. Toumi, A.Murli, On the variational data assimilation problem solving and sensitivity analysis, Journal of Computational Physics, vol. 335, 2017, pp. 311-326. 

\bibitem{PPAM2017} R. Arcucci, L. D'Amore, S. Celestino, G.  Laccetti, A. Murli,  A Scalable Numerical Algorithm for Solving Tikhonov Regularization Problems. In: Parallel Processing and Applied Mathematics. LECTURE NOTES IN COMPUTER SCIENCE, vol. 9574, 2016,  pp. 45-54, HEIDELBERG:SPRINGER, ISBN: 978-3-319-32152-3, ISSN: 0302-9743, doi: 10.1007/978-3-319-32152-3-5


\bibitem{DD_pseudo}
J.K. Baksalary,  O.M. Baksalary, Particular formulae for the Moore-Penrose inverse of a columnwise partitioned matrix, Linear Algebra Appl. 421, 2017, pp. 16-23.

\bibitem{Battistelli} G. Battistelli, L. Chisci,  Stability of consensus extended Kalman filter for distributed state, Automatica
Volume 68, 2016, pp. 169-178

\bibitem{Brankart} Brankart, J.-M., C.-E. Testut, P. Brasseur, and J. Verron, 2003: Implementation of a multivariate data assimilation scheme for isopycnic coordinate ocean models: Application to a 1993-1996 hindcast of the North Atlantic Ocean circulation. J. Geophys. Res., 108, 3074, doi:10.1029/2001JC001198

\bibitem{deeplearning4} M. M. Bronstein, J. Bruna, Y. LeCun, A. Szlam, P. Vandergheynst, Geometric deep learning:
going beyond Euclidean, IEEE SIG PROC MAG, arXiv:1611.08097v2 [cs.CV],  3 May 2017


\bibitem{Brusdal} Brusdal, K., J.M. Brankart, G. Halberstadt, G. Evensen, P. Brasseur, P. van Leeuwen, E. Dombrowsky, and J. Verron, 2003: An evaluation of ensemble based assimilation methods with a layered OGCM. J. Mar. Syst., 40-41, 253-289.
\bibitem{Chan} T. F. Chan, T. P. Mathew, Domain Decomposition algorithms, Acta Numerica, 1994, pp. 61-143.


\bibitem{Cohn} S. E. Cohn, An introduction to estimation theory, J. Meteor. Soc. Japan, 75 (1B), 1997, pp. 257-288.


\bibitem{PPAM2019} L. D'Amore, R. Cacciapuoti  Ab initio Domain Decomposition Approaches for Large Scale Kalman Filter Methods:
a case study to Constrained Least Square Problems, 13th International Conference, PPAM 2019, Bialystok, Poland,
September 8-11, 2019. 10.1007/978-3-030-43222-5, LNCS Vol. 12044, Springer.

\bibitem{arxivPPAM2019} R. Cacciapuoti, L. D'Amore, Parallel framework for Dynamic Domain Decomposition of Data Assimilation problems a case study on Kalman Filter algorithm, arXiv:2203.16535v1 [cs.LG] 9 Jan 2022


\bibitem{CAI} L. D'Amore, V. Mele, D.  Romano, G. Laccetti, Multilevel Algebraic Approach for Performance Analysis of Parallel Algorithms, Computing and Informatics, Vol. 38(4), 2019, pp. 817-850.


\bibitem{DD-DA} L. D'Amore, R. Arcucci, L. Carracciuolo, A. Murli - A Scalable Approach for Variational Data Assimilation, J Sci Comput. vol. 14(61), 2014, pp. 239-257, doi: 10.1007/s10915-014-9824-2


\bibitem{Ricerche} L. D'Amore, R.  Cacciapuoti, A note on domain decomposition approaches for solving 3D variational data assimilation models, 
Ricerche di Matematica, 2019. Vol. 68, N, 2.  doi:10.1007/s11587-019-00432-4


\bibitem{DyDD} L. D'Amore, R.  Cacciapuoti, Parallel framework for dynamic domain decomposition of data assimilation problems: a case study on Kalman Filter algorithm, Comp and Math Methods, 2021. 3:e1145. DOI: 10.1002/cmm4.1145

\bibitem{JSC} L. D'Amore, R. Arcucci, L. Carracciuolo,  A.  Murli,  A scalable approach for Variational Data Assimilation. Journal of Scientific Computing, vol. 61,  2014,  pp. 239-257, ISSN: 0885-7474, doi: 10.1007/s10915-014-9824-2

\bibitem{Daescu} D.N. Daescu, I.M. Navon, Sensitivity analysis in nonlinear variational data assimilation: theoretical aspects and applications, in: Istvan Farago, Zahari Zlatev (Eds.), Advanced Numerical Methods for Complex Environmental Models: Needs and Availability, Bentham Science Publishers, December 2013.



\bibitem{Evensen} G. Evensen, The ensemble Kalman filter: Theoretical formulation and practical implementation, Ocean Dynam. Vol.  53, 2003, pp. 343-367.


\bibitem{least1}
W. Gander, Least squares with a quadratic constraint, Numerische Mathematik, vol. 36, 1980, pp. 291-307.

\bibitem{Gander} M. J. Gander Schwarz methods over the course of time, ETNA, vol. 31, 2008, pp. 228-255.

\bibitem{Hahnel} P. Hahnel, J. Marecek,
J. Monteilb, F. O'Donnchab, Using Deep Learning to Extend the Range of
Air Pollution Monitoring and Forecasting, arXiv:1810.09425v3, 2020


\bibitem{Hamill}T. M.  Hamill,  Ensemble forecasting and data assimilation: two problems with the same solution - In: Predictability of Weather and Climate, T. Palmer, R. Hagedorn
(Eds.) Cambridge University Press, Cambridge, 2006, pp. 124-156.


\bibitem{Hannachi} A. Hannachi, I.T. Jolliffe, D.B. Stephenson, Empirical orthogonal functions and related techniques in atmospheric science: a review, Int. J. 
Climatol. vol. 1152, 2007, pp. 1119-1152.


\bibitem{Haugen} Haugen, V. E., and G. Evensen, 2002: Assimilation of SLA and SST data into an OGCM for the Indian ocean. Ocean Dyn., 52, 133-151.

\bibitem{Petzold} C. Homescu, L. R. Petzold, R. Seban,  Error Estimation for REduced Order Models of Dynamical Systems, UCRL-TR-2011494, December 2003.

\bibitem{H_M} P. L.  Houtekamer,
H. L. Mitchell,  A sequential ensemble Kalman Filter for atmospheric data assimilation. Mon. Weather Rev. vol. 129, 2001, pp. 123-137


\bibitem{concurrency98}   Y.F. Hu, R.J. Blake and D.R. Emerson - An optimal migration algorithm for dynamic load balancing,
Concurrency: Practice \& Experience vol. 10(6), 1998, pp. 467-48.

\bibitem{Hunt} B.R. Hunt, E.J. Kostelich, and I. Szunyogh, Efficient data assimilation for spatio temporal chaos: A local ensemble transform Kalman filter. Physica D vol. 23, 2007, pp. 112-126


\bibitem{Janic} T. Janic, L. Nerger, On Domain Localization in Ensemble-Base Kalman Filter Algorithms, Montly Weather Review, Vol 139, pp. 2046-2060, 2011.
\bibitem{least2}
L. Yong, A feasible interior point algorithm for a class of non negative least squares problems, in Future Computer and
Communication, 2009. FCC'09. International Conference on, 2009, pp. 157-159.


\bibitem{Zang} 
Z. Z. John, C .Moore, Data Assimilation,  Mathematical and Physical Fundamentals of Climate Change, 2015, pp. 291-311.


\bibitem{Kalman}
R. E. Kalman, A New Approach to Linear Filtering and Prediction
Problems, Transaction of the ASME - Journal of Basic Engineering, 1960. pp. 35-45.

\bibitem{Kalnay} E. Kalnay,  Atmospheric Modeling, Data Assimilation and Predictability Cambridge University Press, 2003

\bibitem{Khan} U. A. Khan, Distributing the Kalman Filter for
Large-Scale Systems, IEEE TRANSACTIONS ON SIGNAL PROCESSING, Vol. 56(10), pp. 4919 - 4935,  2008. DOI: 10.1109/TSP.2008.927480


\bibitem{Kepenne}  C.L. Keppenne,  Data assimilation into a primitive-equation model with a
parallel ensemble Kalman filter. Monthly Weather Review vol. 128, 2000, pp. 1971-1981.
\bibitem{Kepenne2002}   C.L. Keppenne,, Rienecker, M.M.,  Initial testing of a massively parallel
ensemble Kalman filter with the Poseidon isopycnal ocean circulation model.
Monthly Weather Review vol. 130, 2002, pp. 2951-2965.

\bibitem{Kim} Y. Kim and H. Bang, Introduction to Kalman Filter and Its Applications, 2018,  DOI: 10.5772/intechopen.80600


\bibitem{LeVeque} J. LeVeque, Randall Numerical Methods for Conservation Laws, Birkhauser Verlag, 1992.


 \bibitem{Lyster} P. M. Lyster, J. Guo, T. Clune, J. W. Larson, The Computational Complexity and Parallel Scalability of Atmospheric Data Assimilation Algorithms, JOURNAL OF ATMOSPHERIC AND OCEANIC TECHNOLOGY, Vol. 21, 2004,  pp. 1689-1700.
 
 


\bibitem{Lions}  P.L. Lions. On the Schwarz alternating method. III. A variant for nonoverlapping subdomains. In Third International Symposium on Domain Decomposition Methods for Partial Differential Equations (Houston, TX, 1989), pp. 202-223. SIAM, Philadelphia, PA, 1990.

\bibitem{PINT}  J.L. Lions, Y. Maday, and G. Turinici, A parareal in time discretization of PDE's, C. R.
Acad. Sci. Paris Ser. I Math. vol. 332, 2001, pp. 661-668.




\bibitem{deeplearning3} Z. Long, Y. Lu, B. Dong,  PDE-NET 2.0: Learning PDEs from data with a numeric symbolic hybrid deep network, arXiv: 1812.04426 [cs.LG, ],  1 Sep 2019.

\bibitem{Lorenc1981}A. Lorenc, A global three dimensional multivariate interpolation scheme, Montly Weather Review, Vol. 109, N. 4, pp. 701-721, 1981.


\bibitem{Meurant} G. Meurant. Domain decomposition methods for partial differential equations on parallel
computers. The International Journal of Supercomputing Applications, vol. 2(4), 1988, pp. 5-12.


\bibitem{Nerger} L. Nerger, Parallel Filter Algorithms for Data Assimilation in Oceanography. Ph.D. Thesis, University of Bremen, Germany, 2003.

\bibitem{Nerger2005} L. Nerger, Hiller, W., Schroter, J. PDAF - the parallel data assimilation framework: experiences with Kalman filtering. In Zwieflhofer, W., Mozdzynski, G. (Eds), Use of High Performance Computing in Meteorology. Proceedings of the 11th ECMWF Workshop. World Scientific, 2005,  pp. 63-83


\bibitem{Nerger2006} L. Nerger, S. Danilov, W. Hiller, and
J. Schroter,  Using sea-level data to constrain a finite-element primitive equation ocean model with a local SEIK
filter, 2006, Ocean Dynamics, vol. 56, pp. 634-649


\bibitem{Nichols}  N. Nichols, Mathematical concepts in data assimilation, in: W. Lahoz, et al. (Eds.), Data Assimilation, Springer, 2010.

\bibitem{Ott} Ott, E., and Coauthors, 2004: A local ensemble Kalman filter for atmospheric data assimilation. Tellus, 56A, 415-428

\bibitem{Quarteroni} A. Quarteroni, F. Saleri, Matematica Numerica, Springer Verlag, 2 edizione, 2000, pp. 109-111.


\bibitem{Rafiee} M. Rafiee, A. Tinka,  J. Thai, A. M. Bayen, Combined State-Parameter Estimation for Shallow Water Equations,  American Control Conference
on O'Farrell Street, San Francisco, CA, USA
June 29 - July 01, 2011.

\bibitem{Rozier} D. Rozier, F. Birol, E. Cosme, P. Brasseur,J. M. Brankart, J. Verron, A Reduced-Order Kalman Filter for Data Assimilation in Physical Oceanography, SIAM REVIEW, Vol. 49(3), 2007, pp. 449-465

\bibitem{Saad} Y. Saad. Iterative methods for sparse linear systems. SIAM, 2003.

\bibitem{Sandu} Nino-Ruiz, Elias D., A. Sandu, X. Deng - A parallel implementation of the ensemble Kalman filter based on modified Cholesky decomposition, Journal of Computational Science, 30, 2019: 100654



\bibitem{Schur}  I. Schur. Gesammelte Abhandlungen. Band II., 1973.

 \bibitem{Schwarz} H.A. Schwarz. Ueber einige abbildungsaufgaben. Journal f¨ur die reine und angewandte Mathematik, vol. 70, 1869, pp. 105-120.


\bibitem{Sherman} 
J. Sherman, J. Morrison Winifred,  Adjustment of an Inverse Matrix Corresponding to Changes in the Elements of a Given  Matrix, Ann. Math. Statist.  Vol. 21, N. 1, pp. 124-127, 1950. DOI: 10.1214/aoms/1177729893

\bibitem{sorenson} 
H. W. Sorenson, Least square estimation:from Gauss to Kalman, IEEE Spectrum, Vol. 7, 1970,  pp. 63-68.

\bibitem{Tang} X. Tang, G. Falco, E. Falletti, L. Lo Presti,  Complexity reduction of the Kalman filter based tracking loops in GNSS receivers. GPS Solut., vol. 21, 2017, pp. 685-699. 


\bibitem{Teruzzi}  A. Teruzzi, P. Di Cerbo, G. Cossarini, E. Pascolo, S. Salon,  Parallel implementation of a data assimilation scheme for operational oceanography: The case of the MedBFM model system
Computers \& Geosciences
Vol. 124,  2019, pp. 103-114.

\bibitem{Tirupachi}  S. Tirupathi, T. Tchrakiana, S. Zhuk, S. McKenna Tirupachi, Shock Capturing Data Assimilation Algorithm for 1D Shallow Water Equations, Advances in Water Resources, Vol. 88, 2016, pp. 198-210. 

\bibitem{Tossavainen} Olli-Pekka Tossavainen, J. Percelay, A. Tinka, Q. Wu,  A. M. Bayen, 
Ensemble Kalman filter based state estimation in 2D shallow water equations using Lagrangian sensing and state augmentation,  Proceedings of the
47th IEEE Conference on Decision and Control
Cancun, Mexico, Dec. 9-11, 2008


\bibitem{Wikle} C.K. Wikle,  N. Cressie, A dimension-reduced approach to space-time Kalman filtering, Biometrika, Vol. 86(4), 1999, pp. 815-829, doi.org/10.1093/biomet/86.4.815.


\bibitem{NEMO} NEMO Release 4.0, Nucleus for European Modelling of the Ocean,  https://www.nemo-ocean.eu/

\bibitem{Karavasilis} V. Karavasilis,  C. Nikou, A. Likas, Visual Tracking by Adaptive Kalman Filtering and Mean Shift. In: Konstantopoulos S., Perantonis S., Karkaletsis V., Spyropoulos C.D., Vouros G. (eds) Artificial Intelligence: Theories, Models and Applications. SETN 2010. Lecture Notes in Computer Science, vol. 6040. Springer, Berlin, Heidelberg, 2010.

\bibitem{Benhamou} E. Benhamou, Kalman filter demystified: from intuition to probabilistic graphical model to real case in financial markets, 	arXiv:1811.11618, Dec. 2018.

\bibitem{Heidari} L. Heidari,  V. Gervais,  M. Le Ravalec,  Hans Wackernagel,  History Matching of Reservoir Models by Ensemble Kalman Filtering: The State of the Art and a Sensitivity Study, Uncertainty Analysis and Reservoir Modeling, 2011, https://doi.org/10.1306/13301418M963486


\bibitem{Lu}F.  Lu, H.  Zeng, Application of Kalman Filter Model in the Landslide Deformation Forecast. Sci Rep 10, 1028 (2020). https://doi.org/10.1038/s41598-020-57881-3

\bibitem{Wu} Y. Wu, Y. Sui and G. Wang, Vision-Based Real-Time Aerial Object Localization and Tracking for UAV Sensing System,  IEEE Access, vol. 5, pp. 23969-23978, 2017, doi: 10.1109/ACCESS.2017.2764419

\bibitem{Zhou} Y. Zhou, D. McLaughlin, D. Entekhabi,  Gene-Hua Crystal Ng, Ensemble Multiscale Filter for Large Nonlinear Data Assimilation Problems, Monthly Weather Review, Vol. 136: Issue 2, pp.678-698.
\end{thebibliography}
\end{document}